\documentclass[11pt,reqno,oneside,a4paper]{article}
\usepackage[a4paper,includeheadfoot,left=25mm,right=25mm,top=00mm,bottom=20mm,headheight=20mm]{geometry}

\usepackage{amssymb,amsmath,amsthm}
\usepackage{xcolor,graphicx}
\usepackage{verbatim}
\usepackage{mathtools}
\usepackage[colorlinks=true, pdfstartview=FitV, linkcolor=blue, citecolor=blue, urlcolor=blue]{hyperref}
\usepackage{titling}
\usepackage{fancyhdr}
\newcommand{\runningtitle}{Running Title}
\newcommand{\runningauthor}{Running Author}
\newtheorem{thm}{Theorem}
\newtheorem{lem}[thm]{Lemma}

\newtheorem{prop}[thm]{Proposition}

\theoremstyle{definition}

\theoremstyle{remark}
\newtheorem{rmk}[thm]{Remark}

\usepackage{scalerel}
\usepackage{stackengine}
\stackMath
\newcommand\reallywidecheck[1]{%
\savestack{\tmpbox}{\stretchto{%
  \scaleto{%
    \scalerel*[\widthof{\ensuremath{#1}}]{\kern-.6pt\bigwedge\kern-.6pt}%
    {\rule[-\textheight/2]{1ex}{\textheight}}
  }{\textheight}%
}{0.5ex}}%
\stackon[1pt]{#1}{\scalebox{-1}{\tmpbox}}%
}

\newcommand\reallywidehat[1]{%
\savestack{\tmpbox}{\stretchto{%
  \scaleto{%
    \scalerel*[\widthof{\ensuremath{#1}}]{\kern-.6pt\bigwedge\kern-.6pt}%
    {\rule[-\textheight/2]{1ex}{\textheight}}
  }{\textheight}%
}{0.5ex}}%
\stackon[1pt]{#1}{\tmpbox}%
}

\NewDocumentEnvironment{MLME}{O{c}m}{%
  \multlined[#1][#2]
}{\endmultlined}
\usepackage{tikz}
\usepackage{pgfplots}

\newcounter{tagtageqn}
\newcommand\tagtagsubeq[2]{{\renewcommand{\thetagtageqn}{#1}\newcommand{\parenteqntag}{#1}\refstepcounter{tagtageqn}#2}}

\newcommand{\NN}{\mathbb N}              
\newcommand{\RR}{\mathbb R}              
\newcommand{\CC}{\mathbb C}              
 \renewcommand{\Im}{\operatorname*{Im}}
\newcommand{\D}{\ensuremath{\,\mathrm{d}}}
\newcommand{\ri}{\ensuremath{\mathrm{i}}}
\newcommand{\re}{\ensuremath{\mathrm{e}}}
\newcommand{\la}{\ensuremath{\lambda}}
\renewcommand{\epsilon}{\varepsilon}
\renewcommand{\geq}{\geqslant}
\renewcommand{\leq}{\leqslant}

\providecommand{\clos}{\operatorname{clos}}
\newcommand{\abs}[1]{\left\lvert#1\right\rvert}

\newcommand{\normp}[2]{\left\lVert#1\right\rVert_{#2}}
\newcommand{\norm}[1]{\left\lVert#1\right\rVert}

\newcommand{\Mspacer}{\;} 
\newcommand{\M}[3]{#1_{#2\Mspacer#3}} 

\providecommand{\bigoh}[1]{\mathcal{O}\left(#1\right)}
\providecommand{\bigohnoscale}[1]{\mathcal{O}(#1)}

\providecommand{\argdot}{{}\cdot{}}

\providecommand{\Lebesgue}{\mathrm{L}} 
\providecommand{\ContinuousSpace}{\mathrm{C}} 

\newcommand{\CSchro}{{\mathcal C_{\mathrm S}}}
\newcommand{\CAiry}{{\mathcal C_3}}
\newcommand{\DAiry}{{\mathcal D_3}}
\newcommand{\CHeat}{{\mathcal C_{\mathrm h}}}
\newcommand{\DHeat}{{\mathcal D_{\mathrm h}}}
\newcommand{\CbiSone}{{\M{\mathcal C}41}}
\newcommand{\CbiStwo}{{\M{\mathcal C}42}}

\newcommand{\what}{\widehat}
\newcommand{\wt}{\widetilde}
\newcommand{\nn}{\nonumber}

\allowdisplaybreaks

\author{%
  Dionyssios Mantzavinos\textsuperscript{*},
  Ravindra Pethiyagoda\textsuperscript{\S} and
  Dave Smith\textsuperscript{\S\textdagger} \\
  \small \textsuperscript{*} University of Kansas, Lawrence KS, USA \\
  \small \textsuperscript{\S} University of Newcastle, Newcastle NSW, Australia \\
  \small \textsuperscript{\textdagger} \texttt{dave.smith@newcastle.edu.au}
}
\title{Interface problems of mixed spatial order}
\renewcommand{\runningauthor}{Mantzavinos, Pethiyagoda \& Smith}
\renewcommand{\runningtitle}{Interface mixed order}
\date{\today}

\begin{document}
\maketitle
\thispagestyle{fancy}

\begin{abstract}
  We solve interface problems on the line between various constant coefficient linear evolution partial differential equations.
  Our prototypical examples are the heat, linear Schr\"odinger, Airy, linearized Korteweg de Vries, and biharmonic Schr\"odinger equations.
  In each problem, one of the listed equations is posed on one spatial half line, and another on the other half line, with appropriate interface conditions.
  These problems are solved by means of novel extensions of Fokas's unified transform method and the explicit solution formulae are evaluated using Filon quadrature.
\end{abstract}


\section{Introduction} \label{sec:Introduction}

In the study of evolution partial differential equations, initial boundary value problems arise naturally in applications where the physical domain has a boundary, such as a half line or a finite interval in one spatial dimension.
Contrary to initial value (Cauchy) problems, which take place on fully infinite or periodic spatial domains (the whole line or the circle in one spatial dimension) and only require the prescription of initial conditions, initial boundary value problems also involve a suitable number of appropriate boundary conditions.
In one spatial dimension, such boundary conditions generally appear in the form of one or more boundary values as given functions of the temporal variable; typical scenarios include those of Dirichlet, Neumann and Robin data, which respectively correspond to the prescription of the solution itself at the boundary, its spatial derivative, and a linear combination thereof.

Due to the presence of a boundary and associated boundary conditions, the analysis of initial boundary value problems is generally more complex than the one of initial value problems, this being the case even at the level of linear equations.
Importantly, the Fourier transform, which is a simple yet powerful method for solving linear evolution equations on the whole space, is no longer available in the case of linear initial boundary value problems posed on the half line, the finite interval, or other domains with a boundary in higher dimensions.
In 1997, Fokas introduced a general method for the solution of linear (and certain nonlinear) initial boundary value problems which is now known as the unified transform method or the Fokas method~\cite{f1997,f2008}.
The unified transform provides the direct analogue of the Fourier transform when the physical domain involves a boundary. Since its inception, it has been developed substantially through a vast number of research works in various directions, as documented in the expository works and conference proceedings volumes~\cite{DTV2014a,fp2015,bountis2022} and the references therein.

A particular direction of interest in connection to the above is the study of interface problems, in which the spatial domain is divided into different regions governed by different modeling equations.
In this context, the classical boundary conditions are replaced by interface conditions, which  reflect the physics of the underlying phenomenon as one transitions from one region to the next one through the respective interface.
For example, in the case of the brain, due to the heterogeneity of the brain tissue (white and grey matter), a simple mathematical model for the growth of certain tumors known as gliomas is one that involves the heat (diffusion) equation but with a discontinuous diffusion coeﬃcient~\cite{sbma2003}.
In this case, the interface conditions can be chosen so that the solution as well as its first partial derivative multiplied by the diffusion coefficient is continuous across each interface.
In~\cite{mpss2016}, this tumor growth interface problem was solved analytically by means of the unified transform in the case of a finite domain consisting of three regions (grey matter---white matter---grey matter).
The main idea behind the analysis of~\cite{mpss2016} was to treat the problem as a chain of initial boundary value problems, despite the fact that the various boundary values involved in the interface conditions were not actually prescribed individually as data.
Then, by combining continuity at the interface with the symmetries of a certain spectral identity that lies in the very core of the unified transform and is known as the global relation, it was possible to eliminate the unknown interface values and thereby obtain an explicit solution in each of the three aforementioned regions.

The unified transform based approach of~\cite{mpss2016} was subsequently employed in~\cite{dps2014} for the solution of more general interface problems associated to heat conduction in one dimensional piecewise homogeneous composite materials.
The technique was then extended to the heat equation on a network with perfect thermal contact~\cite{ss2015}, as well as to dispersive equations: the case of the linear Schr\"odinger and Airy equations \cite{sd2015,dss2016}.
The same method has been adapted to the linear Schr\"odinger equation with a piecewise constant potential~\cite{DS2020a}.
Via a continuum limit of interface problems for piecewise constant coefficient potentials, dissipative equations with smoothly varying coefficients have been solved~\cite{DF2023a,DF2025a}.
Multipoint value problems (a class of interface problems on a rose metric graph) for various dispersive and dissipative evolution equations have been analysed via another adaptation of the unified transform method~\cite{PS2018a}.
Interface problems for any single linear evolution equation with polynomial dispersion relation on any metric graph composed of only finite edges are solved constructively in~\cite{ABS2022a}.
By a similar approach, the linearized KdV equation has been solved on various metric graph domains having semiinfinite edges~\cite{Smi2026a}.
Many of the above mentioned works admit rather general interface conditions, but they all require that the equations on each spatial interval are the same, or at least the spatially leading order terms of each of those equations differ only by a real (or, in the cases of~\cite{ABS2022a,DF2025a}, complex) scalar.

In the present work, we consider interface problems that emerge from the interaction between linear evolution equations of \textit{fundamentally different behavior}.
Indeed, we consider purely dispersive models but of different orders, as well as the interaction of a dissipative and a dispersive model.
More specifically, each of the interface problems we study comprises one of the left half line problems: the linear Schr\"odinger equation
\begin{equation} \label{eqn:LS} \tag{lS}
\begin{aligned}
  \ri w_t + w_{xx} &= 0 & \qquad (x,t) &\in (-\infty,0)\times(0,T), \\
  w(x,0) &= w_0(x) & x &\in(-\infty,0),
\end{aligned}
\end{equation}
the heat equation
\begin{equation} \label{eqn:h} \tag{h}
\begin{aligned}
  w_t - w_{xx} &= 0 & \qquad (x,t) &\in (-\infty,0)\times(0,T), \\
  w(x,0) &= w_0(x) & x &\in(-\infty,0),
\end{aligned}
\end{equation}
or the biharmonic Schr\"odinger equation
\begin{equation} \label{eqn:biS} \tag{biS}
\begin{aligned}
  \ri w_t + w_{xxxx} &= 0 & \qquad (x,t) &\in (-\infty,0)\times(0,T), \\
  w(x,0) &= w_0(x) & x &\in(-\infty,0),
\end{aligned}
\end{equation}
and one of the right half line problems:
the Airy equation
\begin{equation} \label{eqn:Airy} \tag{Airy}
\begin{aligned}
  v_t + v_{xxx} &= 0 & \qquad (x,t) &\in (0,\infty)\times(0,T), \\
  v(x,0) &= v_0(x) & x &\in(0,\infty),
\end{aligned}
\end{equation}
or, for any constant $a\in\RR$, the linearized Korteweg de Vries equation
\begin{equation} \label{eqn:lKdV} \tag{lKdV}
\begin{aligned}
  v_t + v_{xxx} + av_x &= 0 & \qquad (x,t) &\in (0,\infty)\times(0,T), \\
  v(x,0) &= v_0(x) & x &\in(0,\infty),
\end{aligned}
\end{equation}
together with one of the sets of interface conditions:
\tagtagsubeq{IfC}{\label{eqn:IfC}%
\begin{align}
  \label{eqn:IfC.1} \tag{\parenteqntag.1}
  \partial_x^j w(0,t) &= \partial_x^j v(0,t) \eqcolon g_j(t) & t &\in [0,T], \qquad j\in\{0,1\}, \\*
  \label{eqn:IfC.2} \tag{\parenteqntag.2}
  \partial_x^j w(0,t) &= \partial_x^j v(0,t) \eqcolon g_j(t) & t &\in [0,T], \qquad j\in\{0,1,2\}.
\end{align}%
}%
In either of~\eqref{eqn:IfC} the $g_j$ will not be specified data but are interface value functions defined for notational convenience.
Therefore,~\eqref{eqn:IfC.1} specifies that the piecewise function $(w,v)$ is $\ContinuousSpace^1$ across the interface but does not specify its value or derivative there, while~\eqref{eqn:IfC.2} enforces $\ContinuousSpace^2$ across the interface without fixing the value or its derivatives.
In order to state the solutions of such problems, we shall require some further notation.

We define the half line Fourier transforms of functions $f\in \Lebesgue^2(0,\infty)$ and $g\in \Lebesgue^2(-\infty,0)$ by
\begin{align} \label{hlft-def}
  \what f(k) &= \int_0^\infty \re^{-\ri kx} f(x) \D x, & \Im(k) &\leq 0, & & &
  \what g(k) &= \int_{-\infty}^0 \re^{-\ri kx} g(x) \D x, & \Im(k) &\geq 0.
\end{align}
We overload notation here having in mind that the negative or positive half line Fourier transform is evident from the domain of definition of the corresponding function.

Using the principal branches of the root functions
$z^{\frac 1n} = \sqrt{\abs z} \re^{\ri \arg(z)/n}$, for $-\pi < \arg(z) \leq \pi$, for $n\in\NN$,
we define the root type functions
\begin{align}
  \label{eqn:defn.sigma}
  \sigma(\la) &= \re^{ \ri\frac\pi4}\left(\ri\la\right)^{\frac 12} \\
  \label{eqn:defn.rho}
  \rho  (\la) &= \re^{-\ri\frac\pi6}\left(\ri\la\right)^{\frac 13} \\
  \label{eqn:defn.tau}
  \tau  (\la) &= \ri                \left(\ri\la\right)^{\frac 12} \\
  \label{eqn:defn.gamma}
  \gamma(\la) &= \re^{-\ri\frac\pi8}\left(\ri\la\right)^{\frac 14}.
\end{align}
These functions are displayed as maps of quadrants of the complex $\la$ plane in figure~\ref{fig:sigma} for $\sigma$, figure~\ref{fig:rho} for $\rho$, figure~\ref{fig:tau} for $\tau$, and figure~\ref{fig:gamma} for $\gamma$.
For problems involving~\eqref{eqn:lKdV}, let $R>0$ be sufficently large that $\nu$ satisfying
\[
  \nu(k)^3-a\nu(k)=k^3
  \quad\mbox{and}\quad
  \frac{\nu(k)}{k} \xrightarrow[k\to\infty]{} 1
\]
is a biholomorphism from $\CC\setminus B(0,R)$ to the exterior of a suitable compact region.
We define the complex contours, all with positive orientation, as displayed in figure~\ref{fig:contours}
\begin{figure}
  \centering
  \includegraphics[scale=0.7]{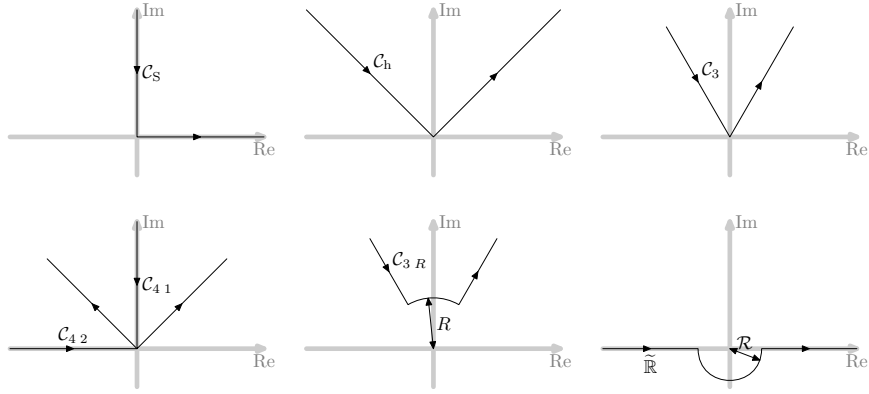}
  \caption{Integration contours.}
  \label{fig:contours}
\end{figure}
\begin{align*}
  \CSchro  &= \partial \left( \mbox{first quadrant} \right), &
  \CHeat   &= \partial \left\{ k \in \CC: \frac \pi 4 < \arg(k) < \frac{3\pi}{4} \right\}, \\
  \CAiry   &= \partial \left\{ k \in \mathbb C: \frac \pi 3 < \arg(k) < \frac{2\pi}{3} \right\}, \\
  \CbiSone &= \partial \left\{k\in\mathbb C: \frac{\pi}{4} < \arg(k) < \frac{\pi}{2} \right\}, &
  \CbiStwo &= \partial \left\{k\in\mathbb C: \frac{3\pi}{4} < \arg(k) < \pi \right\}, \\
  \M{\mathcal C}3R &= \partial \left\{ k \in \mathbb C: \frac \pi 3 < \arg(k) < \frac{2\pi}{3}, \; \abs k > R \right\}, &
  \wt\RR   &= \partial \left( \CC^+ \cup B(0,\mathcal R) \right),
\end{align*}
for arbitrary positive $\mathcal R$ (or, for problems involving~\eqref{eqn:lKdV}, $\mathcal R>R$).
Finally, we denote the primitive cube root of unity by $\alpha = \re^{\ri\frac{2\pi}{3}}$.

\begin{prop} \label{prop:lS-Airy}
  The solution of interface problem~\eqref{eqn:LS},~\eqref{eqn:Airy},~\eqref{eqn:IfC.1} is
  \begin{multline*}
    w(x, t)
    =
    \frac{1}{2\pi} \int_{\RR} \re^{-\ri kx-\ri k^2 t} \, \what w_0(-k) \D k
    -
    \frac{1}{2\pi} \int_{\CSchro} \re^{-\ri kx-\ri k^2 t} \, \what w_0(k) \D k
    \\
    +
    \frac{1}{2\pi} \int_{\wt{\RR}} \re^{-\ri\sigma(k)x+\ri k t}
    \,
    \frac{\left(1-\alpha^2\right)\rho(k) \, \what w_0(\sigma(k)) - \what v_0(\rho(k)) + \what v_0(\alpha^2 \rho(k))}{\left(\alpha-1\right)\rho(k) \left[\rho(k) + \alpha^2 \sigma(k)\right]}
    \D k
  \end{multline*}
  and
  \begin{multline*}
    v(x, t)
    =
    \frac{1}{2\pi} \int_{\RR} \re^{\ri kx+\ri k^3 t} \, \what v_0(k) \D k
    +
    \frac{1}{2\pi} \int_{\CAiry} \re^{\ri kx+\ri k^3 t} \left[\alpha \, \what v_0(\alpha k) + \alpha^2 \, \what v_0(\alpha^2 k)\right] \D k
    \\
    +
    \frac{1}{2\pi} \int_{\wt{\RR}} \re^{\ri\alpha \rho(k) x+\ri k t}
    \,
    \frac{\left(1-\alpha^2\right)\rho(k) \, \what w_0(\sigma(k)) - \what v_0(\rho(k)) + \what v_0(\alpha^2 \rho(k))}{\left(\alpha-1\right)\rho(k) \left[\rho(k) + \alpha^2 \sigma(k)\right]}
    \D k.
  \end{multline*}
\end{prop}

\begin{prop} \label{prop:lS-lKdV}
  The solution of interface problem~\eqref{eqn:LS},~\eqref{eqn:lKdV},~\eqref{eqn:IfC.1} is
  \begin{multline*}
    2\pi w(x, t)
    =
    \int_{\RR} \re^{-\ri kx-\ri k^2 t} \, \what w_0(-k) \D k
    -
    \int_{\CSchro} \re^{-\ri kx-\ri k^2 t} \, \what w_0(k) \D k
    \\
    +
    \int_{\wt{\RR}} \re^{-\ri\sigma(\lambda)x+\ri\lambda t}
    \,
    \frac{\left( \nu(\rho(\lambda))-\nu(\alpha^2\rho(\lambda)) \right) \what w_0(\sigma(\lambda)) - \what v_0(\nu(\rho(\lambda))) + \what v_0(\nu(\alpha^2\rho(\lambda)))}{\Delta(\lambda)}
    \D \lambda,
  \end{multline*}
  and
  \begin{multline*}
    2\pi v(x, t)
    =
    \int_{\mathbb R} \re^{\ri kx+\ri (k^3-ak) t} \, \what v_0(k) \D k
    \\
    +
    \int_{\M{\mathcal C}3R} \re^{\ri \nu(k)x+\ri k^3 t} \frac{(\nu(\alpha k)-\nu(k))\what v_0(\nu(\alpha^2k)) - (\nu(\alpha^2k)-\nu(k))\what v_0(\nu(\alpha k))}{\nu(\alpha^2k)-\nu(\alpha k)} \nu'(k) \D k
    \\
    +
    \int_{\wt\RR} \re^{\ri\nu(\alpha \rho(\lambda)) x+\ri\lambda t} N(\la)
    \,
    \frac{\left( \nu(\rho(\lambda))-\nu(\alpha^2\rho(\lambda)) \right) \what w_0(\sigma(\lambda)) - \what v_0(\nu(\rho(\lambda))) + \what v_0(\nu(\alpha^2\rho(\lambda)))}{\Delta(\lambda)}
    \D\lambda.
  \end{multline*}
\end{prop}

It would be a straightforward application of Duhamel's principle to admit forcing terms on the right sides of the PDE of each of~\eqref{eqn:LS} and~\eqref{eqn:lKdV}, leading to formulae that can provide the starting point for the analysis of the interface problem for the nonlinear Schr\"odinger and Korteweg de Vries equations via contraction mapping techniques.

\begin{prop} \label{prop:heat-Airy}
  The solution of interface problem~\eqref{eqn:h},~\eqref{eqn:Airy},~\eqref{eqn:IfC.1} is
  \begin{multline*}
    w(x, t)
    =
    \frac1{2\pi} \int_{\RR} \re^{-\ri kx-k^2 t} \, \what w_0(-k) \D k
    -
    \frac1{2\pi} \int_{\CHeat} \re^{-\ri kx-k^2 t} \, \what w_0(k) \D k
    \\
    +
    \frac1{2\pi} \int_{\wt{\RR}} \re^{-\ri\tau(\lambda)x+\ri\lambda t}
    \,
    \frac{\left(1-\alpha^2\right)\rho(\lambda) \, \what w_0(\tau(\lambda)) + \ri \what v_0(\rho(\lambda)) - \ri \what v_0(\alpha^2 \rho(\lambda))}{\Delta(\lambda)}
    \D \lambda,
  \end{multline*}
  and
  \begin{multline*}
    v(x, t)
    =
    \frac1{2\pi} \int_{\RR} \re^{\ri kx+\ri k^3 t} \, \what v_0(k)\D k
    +
    \frac1{2\pi} \int_{\CAiry} \re^{\ri kx+\ri k^3 t} \left[\alpha \, \what v_0(\alpha k) + \alpha^2 \, \what v_0(\alpha^2 k)\right] \D k
    \\
    +
    \frac1{2\pi} \int_{\wt{\RR}} \re^{\ri\alpha \rho(\lambda) x+\ri\lambda t}
    \,
    \frac{\left(1-\alpha^2\right)\rho(\lambda) \, \what w_0(\tau(\lambda)) + \ri \what v_0(\rho(\lambda)) - \ri \what v_0(\alpha^2 \rho(\lambda))}{\Delta(\lambda)}
    \D\lambda.
  \end{multline*}
\end{prop}

\begin{prop} \label{prop:biS-Airy}
  The solution of interface problem~\eqref{eqn:biS},~\eqref{eqn:Airy},~\eqref{eqn:IfC.2} is
  \begin{align*}
    2\pi w(x, t)
    \hspace{-3em} &\hspace{3em} =
    \int_{\RR} \re^{-\ri kx+\ri k^4t} \, \what w_0(-k) \D k
    -
    \int_{\CbiSone} \re^{-\ri kx+\ri k^4t} \left[\left(1+i\right) \what w_0(\ri k) - \ri \, \what w_0(k) \right] \D k
    \nn\\
    &\quad
    -
    \int_{\CbiStwo} \re^{-\ri kx+\ri k^4t} \left[\left(1-\ri\right) \what w_0(-\ri k) + \ri \, \what w_0(k) \right] \D k
    \nn\\
    &\quad
    -
    \frac{1-\ri}{2}
    \int_{\widetilde{\RR}} \re^{\gamma(\lambda) x+\ri\lambda t}
    \left(\alpha-1\right) \rho(\lambda)
    \left[\frac{\rho(\la)}{\gamma(\la)}-\alpha^2\right]
    \left[\what w_0(-\gamma(\lambda)) - \what w_0(\ri\gamma(\lambda))\right]
    \frac{\D\lambda}{\Delta(\la)}
    \nn\\
    &\quad
    -
    \frac{1-\ri}{2}
    \int_{\widetilde{\RR}} \re^{\gamma(\lambda) x+\ri\lambda t}
    \bigg\{
    \left[
    \left(1+\ri\right) \rho(\lambda) \left(\gamma(\lambda) + \rho(\lambda)\right) - \left(3-\ri\right) \gamma(\lambda)^2
    \right]
    \what v_0(\alpha^2 \rho(\lambda))
    \nn\\*
    &\hspace*{8em}
    -
    \left[
    \alpha \left(1+\ri\right) \rho(\lambda) \left(\alpha \gamma(\lambda)+\rho(\lambda)\right) -\left(3-\ri\right)\gamma(\lambda)^2
    \right]
    \what v_0(\rho(\lambda))
    \bigg\}
    \frac{\D\lambda}{\Delta(\la)}
    \nn\\
    &\quad
    -
    \frac{1+\ri}{2}
    \int_{\widetilde{\RR}} \re^{\ri\gamma(\lambda) x+\ri\lambda t}
    \left(\alpha-1\right) \rho(\lambda)
    \left[\frac{\ri\rho(\la)}{\gamma(\la)}-\alpha^2\right]
    \left[\what w_0(-\gamma(\lambda)) - \what w_0(\ri\gamma(\lambda))\right]
    \frac{\D\lambda}{\Delta(\la)}
    \nn\\
    &\quad
    -
    \frac{1+\ri}{2}
    \int_{\widetilde{\RR}} \re^{\ri\gamma(\lambda) x+\ri\lambda t}
    \bigg\{
    \left[
    \alpha \left(1-\ri\right) \rho(\lambda) \left(\rho(\lambda)-\ri\alpha \gamma(\lambda)\right) + \left(3+\ri\right)\gamma(\lambda)^2
    \right]
    \what v_0(\rho(\lambda))
    \nn\\*
    &\hspace*{8em}
    -
    \left[
    \left(1-\ri\right)\rho(\lambda) \left(\rho(\lambda)-\ri\gamma(\lambda)\right)
    +
    \left(3+\ri\right) \gamma(\lambda)^2
    \right]
    \what v_0(\alpha^2 \rho(\lambda))
    \bigg\}
    \frac{\D\lambda}{\Delta(\la)},
    \label{eqn:biSAiry.v}
  \end{align*}
  and
  \begin{multline*}
    2\pi v(x, t)
    =
    \int_{\RR} \re^{\ri kx+\ri k^3 t} \, \what v_0(k) \D k
    +
    \int_{\CAiry} \re^{\ri kx+\ri k^3 t} \left[\alpha \, \what v_0(\alpha k) + \alpha^2 \, \what v_0(\alpha^2 k)\right] \D k
    \\
    +
    \left(\alpha^2-1\right)
    \int_{\widetilde{\RR}} \re^{\ri\alpha \rho(\lambda) x+\ri\lambda t}
    \frac{\rho(\lambda)}{\Delta(\la)}
    \left[\what w_0(\ri\gamma(\lambda)) - \what w_0(-\gamma(\lambda))\right] \D\lambda
    \\
    +
    \int_{\widetilde{\RR}} \re^{\ri\alpha \rho(\lambda) x+\ri\lambda t}
    \frac{\gamma(\lambda)}{\Delta(\la)}
    \Big\{\left[\alpha^2 \left(1+\ri\right) \rho(\lambda) - 2\gamma(\lambda)\right] \what v_0(\rho(\lambda))
    -
    \left[\left(1+\ri\right) \rho(\lambda) - 2\gamma(\lambda)\right]
    \what v_0(\alpha^2 \rho(\lambda))
    \Big\} \D\lambda.
  \end{multline*}
\end{prop}


\subsection*{Layout of paper}
In~\S\ref{sec:lSAiry} we prove proposition~\ref{prop:lS-Airy} for the interface problem between the linear Schr\"odinger and Airy equations, demonstrating the novel adaptations to the unified transform method required for problems of mixed spatial order.
Further extensions are required for a problem in which the partial differential equations have lower order terms; the example of the linear Schr\"odinger and linearized Korteweg de Vries is expounded in~\S\ref{sec:lSlKdV} with a proof of proposition~\ref{prop:lS-lKdV}.
It is not necessary that both evolution equations be dispersive, as shown in~\S\ref{sec:heatAiry}, where the heat and Airy problem of proposition~\ref{prop:heat-Airy} is solved.
To demonstrate that the methods are robust to problems of still higher spatial order, the biharmonic Schr\"odinger and Airy problem of proposition~\ref{prop:biS-Airy} is solved in~\S\ref{sec:bihAiry}.
To conclude, we present in~\S\ref{sec:numerics} numerical evaluations of the formulae in propositions~\ref{prop:lS-Airy}--\ref{prop:biS-Airy}, along with a description of the numerical methods used.

\section{Linear Schr\texorpdfstring{\"o}{o}dinger and Airy} \label{sec:lSAiry}

We will prove proposition~\ref{prop:lS-Airy} by solving the interface problem~\eqref{eqn:LS},~\eqref{eqn:Airy},~\eqref{eqn:IfC.1}, in which $w_0$ and $v_0$ are given initial data, but the interface values $g_j$ are unknown functions so denoted for convenience only.
The decision to include two interface conditions can be regarded as reflecting the fact that, when posed independently on their respective half lines, the linear Schr\"odinger equation and the Airy equation each require one boundary condition at $x=0$.

The half line Fourier transform of a function $f\in \Lebesgue^2(0, \infty)$ is defined by equation~\eqref{hlft-def}, and the inversion is obtained via the usual whole line Fourier transform inversion theorem as
\begin{equation}
  f(x) = \frac{1}{2\pi} \int_{\mathbb R} \re^{\ri kx} \what f(k) \D k, \quad x>0.
\end{equation}
Note that $\what f$ is holomorphic in the lower half of the complex $k$-plane by a standard Paley Wiener type theorem (e.g. see~\cite[theorem~7.2.4]{Str1994a}).

We observe that the function $u(x, t) := w(-x, t)$ satisfies the positive half line problem
\begin{subequations} \label{eqn:LSAiryPositive}
\begin{align}
  \label{eqn:LSAiryPositive.LSPDE} \tag{\theparentequation.PDE}
  \ri u_t + u_{xx} &= 0 & (x,t) &\in (0,\infty)\times(0,T), \\
  \label{eqn:LSAiryPositive.LSInC} \tag{\theparentequation.InC}
  u(x,0) &= u_0(x) := w_0(-x) & x &\in(0,\infty), \\
  \label{eqn:LSAiryPositive.LSIfC} \tag{\theparentequation.IfC}
  (-1)^j\partial_x^j u(0,t) &= \partial_x^j v(0,t)  = g_j(t) & t &\in [0,T], \qquad j\in\{0,1\},
\end{align}
\end{subequations}
where we once again emphasize that the functions $g_0(t)$ and $g_1(t)$ are unknown.
Thus, the original interface problem can be replaced by~\eqref{eqn:Airy} and~\eqref{eqn:LSAiryPositive}.

Under the assumption of sufficiently smooth data that decay fast enough at infinity (e.g. in the Schwartz class of functions), the initial boundary value problem for the linear Schr\"odinger equation on the positive half line admits the unified transform solution representation~\cite{Fok2002a}
\begin{equation} \label{utm-u}
  2\pi u(x, t)
  =
  \int_{\mathbb R} \re^{\ri kx-\ri k^2 t} \, \what u_0(k) \D k
  -
  \int_{\CSchro} \re^{\ri kx-\ri k^2 t} \, \what u_0(-k) \D k
  +
  \int_{\CSchro} \re^{\ri kx-\ri k^2 t} \, 2k \, \wt g_0(\ri k^2, t) \D k
\end{equation}
where the temporal transform
\begin{equation} \label{tilde-def}
  \wt g_j(\xi, t) := \int_0^t \re^{\xi \tau} g_j(\tau) \D\tau
\end{equation}
and the contour of integration $\CSchro$ is the positively oriented boundary of the first quadrant of the complex $k$ plane.

Similarly, the initial boundary value problem for the Airy equation on the positive half line admits the unified transform solution representation~\cite{Fok2002a,Fok2008a}
\begin{multline} \label{utm-v}
  2\pi v(x, t)
  =
  \int_{\mathbb R} \re^{\ri kx+\ri k^3 t} \, \what v_0(k) \D k
  +
  \int_{\CAiry} \re^{\ri kx+\ri k^3 t} \left[\alpha \, \what v_0(\alpha k) + \alpha^2 \, \what v_0(\alpha^2 k)\right] \D k
  \\
  -
  \int_{\CAiry} \re^{\ri kx+\ri k^3 t} \, 3k^2 \, \wt g_0(-\ri k^3, t) \D k
\end{multline}
where $\alpha = \re^{\ri\frac{2\pi}{3}}$ and the contour of integration $\CAiry$ denotes the positively oriented boundary of the region
\begin{equation*}
  \DAiry := \left\{ k \in \mathbb C: \frac \pi 3 < \arg(k) < \frac{2\pi}{3} \right\}.
\end{equation*}
In addition to the representations~\eqref{utm-u} and~\eqref{utm-v}, the following spectral identities, known in the unified transform terminology as global relations, are also valid:
\begin{align}
  \re^{\ri k^2 t} \, \what u(k, t)
  &=
  \what u_0(k) + i \, \wt g_1(\ri k^2, t) + k \, \wt g_0(\ri k^2, t), & \Im(k) &\leq 0,
  \label{gr-ls-0}
  \\
  \re^{-\ri k^3 t} \, \what v(k, t)
  &=
  \what v_0(k) + \wt g_2(-\ri k^3, t) + \ri k \, \wt g_1(-\ri k^3, t) - k^2 \, \wt g_0(-\ri k^3, t), & \Im(k) &\leq 0.
  \label{gr-airy-0}
\end{align}

Making the changes of variable $k = \sigma(\lambda) := \re^{\ri\frac\pi4}\left(\ri\lambda\right)^{\frac 12}$ and $k=-\sigma(\lambda)$ (both of which imply $\lambda = -k^2$ and a branch cut along the positive imaginary axis of the complex $\lambda$-plane) in the global relation \eqref{gr-ls-0}, we obtain the pair of identities
\begin{align}
  \re^{-\ri\lambda t} \, \what u(\sigma(\lambda), t)
  &=
  \what u_0(\sigma(\lambda)) + \ri \, \wt g_1(-\ri\lambda, t) + \sigma(\lambda) \, \wt g_0(-\ri\lambda, t), & \tfrac\pi2<&\arg(\lambda)\leq\pi \mbox{ or } \lambda=0,
  \label{gr-ls-1}
  \\
  \re^{-\ri\lambda t} \, \what u(-\sigma(\lambda), t)
  &=
  \what u_0(-\sigma(\lambda)) + i \, \wt g_1(-\ri\lambda, t) - \sigma(\lambda) \, \wt g_0(-\ri\lambda, t), & -\pi\leq&\arg(\lambda)\leq\tfrac\pi2 \mbox{ or } \lambda=0,
  \label{gr-ls-2}
\end{align}
where, as shown in figure~\ref{fig:sigma}, the regions of validity are due to the fact that $\Im(\sigma(\lambda)) \geq 0$ is equivalent to $\Im(\lambda) \leq 0$.
Similarly, making the changes of variable $k = \rho(\lambda) := \re^{-\ri\frac\pi6}(\ri\lambda)^{\frac 13}$, $k = \alpha \rho(\lambda)$ and $k = \alpha^2 \rho(\lambda)$ (all of which imply $\lambda = k^3$ and a branch cut along the positive imaginary axis of the complex $\lambda$ plane), we obtain from global relation~\eqref{gr-airy-0} the pair of identities
\begin{align}
  \re^{-\ri\lambda t} \, \what v(\rho(\lambda), t)
  &=
  \what v_0(\rho(\lambda)) + \wt g_2(-\ri \lambda, t) + \ri \rho(\lambda) \, \wt g_1(-\ri\lambda, t) - \rho(\lambda)^2 \, \wt g_0(-\ri\lambda, t),
  \label{gr-airy-1}
  \\
\intertext{valid on $\tfrac{-3\pi}2<\arg(\lambda)\leq 0$ and at $\la=0$,}
  \re^{-\ri\lambda t} \, \what v(\alpha^2 \rho(\lambda), t)
  &=
  \what v_0(\alpha^2 \rho(\lambda)) + \wt g_2(-\ri \lambda, t) + \ri \alpha^2 \rho(\lambda) \, \wt g_1(-\ri\lambda, t) - \alpha \rho(\lambda)^2 \, \wt g_0(-\ri\lambda, t),
  \label{gr-airy-2}
\end{align}
valid on $-\pi\leq\arg(\lambda)\leq \tfrac\pi2$ and at $\la=0$,
where we have used the fact that $\alpha^4 = \alpha$ and have determined the relevant regions of validity due to the transformations of the region $\Im(k) \leq 0$ under the aforementioned changes of variable, as shown in figures~\ref{fig:sigma} and~\ref{fig:rho}.

\begin{figure}
  \centering
  \includegraphics[scale=0.7]{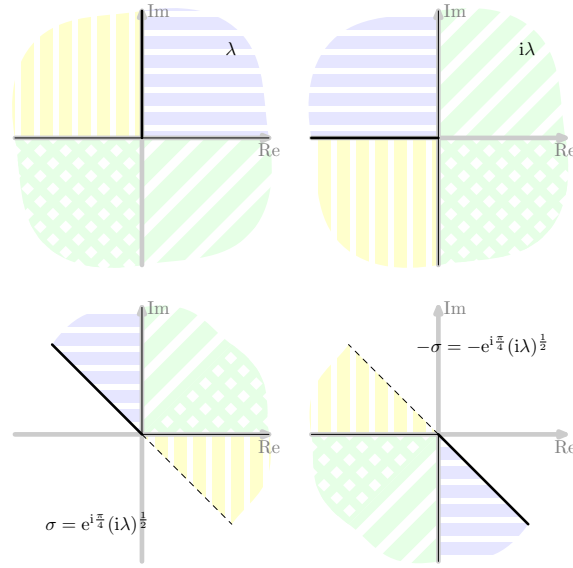}
  \caption{The locus of $\pm \sigma(\lambda)$.}
  \label{fig:sigma}
\end{figure}

\begin{figure}
  \centering
  \includegraphics[scale=0.7]{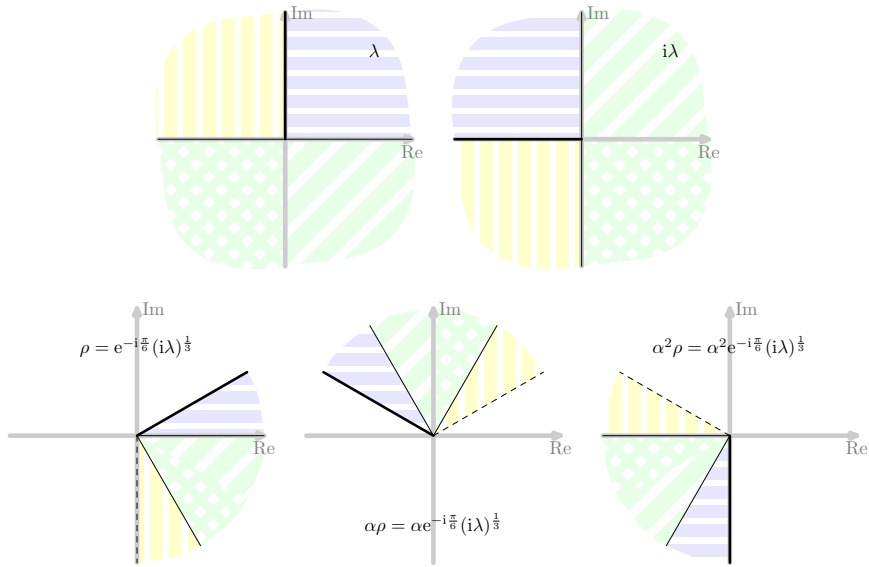}
  \caption{The locus of $\alpha^j \rho(\lambda)$, $j=0, 1, 2$.}
  \label{fig:rho}
\end{figure}

\begin{rmk}
  Under the change of variable $k = \alpha \rho(\lambda)$, the global relation~\eqref{gr-airy-0} provides the supplementary identity
  \begin{equation} \label{gr-airy-3}
    \re^{-\ri\lambda t} \, \what v(\alpha \rho(\lambda), t)
    =
    \what v_0(\alpha \rho(\lambda)) + \wt g_2(-\ri \lambda, t) + \ri \alpha \rho(\lambda) \, \wt g_1(-\ri\lambda, t) - \alpha^2 \rho(\lambda)^2 \, \wt g_0(-\ri\lambda, t), \quad \textnormal{Im}\left(\rho(\lambda)\right) \leq 0,
  \end{equation}
  which, as displayed in figure~\ref{fig:rho}, is only valid for $\lambda=0$ and so is not useful to us.
\end{rmk}

Among the four identities \eqref{gr-ls-1}, \eqref{gr-ls-2}, \eqref{gr-airy-1} and \eqref{gr-airy-2}, the latter three are all valid for $\text{Im}(\lambda) \leq 0$.
Therefore, we can combine them into a $3\times3$ linear system for the transforms $\wt g_0(-\ri\lambda, t)$, $\wt g_1(-\ri\lambda, t)$,  $\wt g_2(-\ri\lambda, t)$ of the unknown boundary values of the solution to the Airy equation, namely
\begin{equation} \label{lin-sys}
  \begin{pmatrix}
    \sigma(\lambda) &1 &0
    \\
    \rho(\lambda)^2 &\rho(\lambda) &1
    \\
    \alpha \rho(\lambda)^2 &\alpha^2 \rho(\lambda) &1
  \end{pmatrix}
  \begin{pmatrix}
    -\wt g_0(-\ri\lambda, t)
    \\
    \ri\, \wt g_1(-\ri\lambda, t)
    \\
    \wt g_2(-\ri\lambda, t)
  \end{pmatrix}
  =
  \re^{-\ri\lambda t}
  \begin{pmatrix}
    \what u(-\sigma(\lambda), t) \\
    \what v(\rho(\lambda), t) \\
    \what v(\alpha^2 \rho(\lambda), t)
  \end{pmatrix}
  -
  \begin{pmatrix}
    \what u_0(-\sigma(\lambda)) \\
    \what v_0(\rho(\lambda)) \\
    \what v_0(\alpha^2 \rho(\lambda))
  \end{pmatrix},
  \quad
  \text{Im}(\lambda) \leq 0.
\end{equation}
Solving this system, we obtain, in particular, the following expression for $\wt g_0(-\ri\lambda, t)$:
\begin{equation} \label{g0-sol}
  \begin{aligned}
    \Delta(\lambda) \, \wt g_0(-\ri\lambda, t)
    &=
    \left[\left(1-\alpha^2\right)\rho(\lambda) \, \what u_0(-\sigma(\lambda)) - \what v_0(\rho(\lambda)) + \what v_0(\alpha^2 \rho(\lambda))\right]
    \\
    &\quad
    -\re^{-\ri\lambda t} \left[\left(1-\alpha^2\right)\rho(\lambda) \, \what u(-\sigma(\lambda), t) - \what v(\rho(\lambda), t) + \what v(\alpha^2 \rho(\lambda), t)\right],
  \end{aligned}
  \quad
  \text{Im}(\lambda) \leq 0,
\end{equation}
where $\Delta$ denotes the determinant of the $3\times 3$ matrix on the left side of \eqref{lin-sys};
\begin{equation} \label{eqn:lSAiry.Delta}
  \Delta(\lambda) = \left(\alpha-1\right) \rho(\lambda) \left[\rho(\lambda) + \alpha^2 \sigma(\lambda)\right].
\end{equation}
Note that a necessary condition for $\rho(\lambda) + \alpha^2 \sigma(\lambda)$ to vanish is that $\rho(\lambda)^6 = \left(-\alpha^2 \sigma(\lambda)\right)^6$ i.e. $\lambda^2 = -\lambda^3$ i.e. $\lambda=0$ or $\lambda=-1$.
However, according to definitions~\eqref{eqn:defn.sigma} and \eqref{eqn:defn.rho}, $\Delta(-1) = \left(\alpha-1\right) \re^{-\ri\pi/3} \left( \re^{-\ri\pi/3} + \re^{-\ri2\pi/3} \right) = 3 \neq 0$ so the only actual zero of $\Delta$ is $\lambda=0$, at which value \eqref{g0-sol} still makes sense.

In order to employ \eqref{g0-sol} in the solution representations \eqref{utm-u} and \eqref{utm-v}, we first need to apply a suitable change of variables in the relevant integrals of these representations.
Let us begin with the last term in \eqref{utm-u}.
Both changes of variable $k=\sigma(\lambda)$ and $k=-\sigma(\lambda)$ imply $\lambda=-k^2$ and hence turn $\wt g_0(ik^2, t)$ into $\wt g_0(-i\lambda, t)$, as desired.
Furthermore, if $k\in\CSchro$ then certainly $\lambda =-k^2 \in \mathbb R$.
However, the correct change of variables is the one such that if $\lambda \in \mathbb R$ then $k\in\CSchro$.
Hence, based on definition~\eqref{eqn:defn.sigma} (see also figure~\ref{fig:sigma}), we should let $k = \sigma(\lambda)$ in the last integral of \eqref{utm-u}, as this transformation indeed maps $\lambda \in \mathbb R$ to $k\in\CSchro$ (on the other hand, $k = - \sigma(\lambda)$ maps $\lambda \in \mathbb R$ to the boundary of the third quadrant of the complex $k$ plane).
Then,~\eqref{utm-u} becomes
\begin{equation} \label{utm-u2}
  2\pi u(x, t)
  =
  \int_{\mathbb R} \re^{\ri kx-\ri k^2 t} \, \what u_0(k) \D k
  -
  \int_{\CSchro} \re^{\ri kx-\ri k^2 t} \, \what u_0(-k) \D k
  +
  \int_{\mathbb R} \re^{\ri\sigma(\lambda)x+\ri\lambda t}  \, \wt g_0(-\ri\lambda, t) \D\lambda.
\end{equation}
Similarly, in view of definition~\eqref{eqn:defn.rho} (see also figure~\ref{fig:rho}), the last term in \eqref{utm-v} requires the change of variable $k = \alpha \rho(\lambda)$, which maps $\lambda \in \mathbb R$ to $k\in\CAiry$ (while $k = \rho(\lambda)$ and $k = \alpha^2 \rho(\lambda)$ map $\lambda \in \mathbb R$ to the boundary of the sixth and fourth sextant of the complex $k$ plane, respectively).
Hence, \eqref{utm-v} becomes
\begin{multline} \label{utm-v2}
  2\pi v(x, t)
  =
  \int_{\RR} \re^{\ri kx+\ri k^3 t} \, \what v_0(k) \D k
  +
  \int_{\CAiry} \re^{\ri kx+\ri k^3 t} \left[\alpha \, \what v_0(\alpha k) + \alpha^2 \, \what v_0(\alpha^2 k)\right] \D k
  \\
  +
  \int_{\RR} \re^{\ri\alpha \rho(\lambda) x+\ri\lambda t} \,   \wt g_0(-\ri\lambda, t) \D \lambda.
\end{multline}

By Cauchy's theorem, we deform the path of integration of the integrals in~\eqref{utm-u2} and~\eqref{utm-v2} that involve $\wt g_0(-\ri\lambda, t)$ locally around $\lambda=0$ from $\RR$ to the contour $\wt{\RR}$ that bypasses $\lambda=0$ via a semicircle lying in the lower half of the complex $\lambda$ plane.
Then, using~\eqref{g0-sol}, we substitute for $\wt g_0(-\ri\lambda, t)$ to obtain
\begin{multline} \label{utm-u3}
  2\pi u(x, t)
  =
  \int_{\RR} \re^{\ri kx-\ri k^2 t} \, \what u_0(k) \D k
  -
  \int_{\CSchro} \re^{\ri kx-\ri k^2 t} \, \what u_0(-k) \D k
  \\
  +
  \int_{\wt{\RR}} \re^{\ri\sigma(\lambda)x+\ri\lambda t}
  \,
  \frac{\left(1-\alpha^2\right)\rho(\lambda) \, \what u_0(-\sigma(\lambda)) - \what v_0(\rho(\lambda)) + \what v_0(\alpha^2 \rho(\lambda))}{\Delta(\lambda)}
  \D \lambda
  \\
  -
  \int_{\wt{\RR}} \re^{\ri\sigma(\lambda)x}  \,
  \frac{\left(1-\alpha^2\right)\rho(\lambda) \, \what u(-\sigma(\lambda), t) - \what v(\rho(\lambda), t) + \what v(\alpha^2 \rho(\lambda), t)}{\Delta(\lambda)}
  \D\lambda
\end{multline}
and
\begin{multline} \label{utm-v3}
  2\pi v(x, t)
  =
  \int_{\RR} \re^{\ri kx+\ri k^3 t} \, \what v_0(k)\D k
  +
  \int_{\CAiry} \re^{\ri kx+\ri k^3 t} \left[\alpha \, \what v_0(\alpha k) + \alpha^2 \, \what v_0(\alpha^2 k)\right] \D k
  \\
  +
  \int_{\wt{\RR}} \re^{\ri\alpha \rho(\lambda) x+\ri\lambda t}
  \,
  \frac{\left(1-\alpha^2\right)\rho(\lambda) \, \what u_0(-\sigma(\lambda)) - \what v_0(\rho(\lambda)) + \what v_0(\alpha^2 \rho(\lambda))}{\Delta(\lambda)}
  \D\lambda
  \\
  -
  \int_{\wt{\RR}} \re^{\ri\alpha \rho(\lambda) x}
  \,
  \frac{\left(1-\alpha^2\right)\rho(\lambda) \, \what u(-\sigma(\lambda), t) - \what v(\rho(\lambda), t) + \what v(\alpha^2 \rho(\lambda), t)}{\Delta(\lambda)}\D\lambda.
\end{multline}

In the usual procedure of the unified transform method, at the stage analogous to this, one would employ Jordan's lemma to remove the contributions of the $\what u(\argdot,t),\what v(\argdot,t)$ dependent integrals (that is, the final integrals over $\wt{\RR}$) from each of equations~\eqref{utm-u3} and~\eqref{utm-v3}.
Indeed, we shall arrive at a similar result.
However, as these integrals have the exponential kernels $\re^{\ri\sigma(\lambda)x}$ and $\re^{\ri\alpha\rho(\lambda)x}$, which do not match the usual $\re^{\ri\lambda x}$ appearing in Jordan's lemma even to leading order, we shall make bespoke arguments for each.

\begin{rmk}
  The deformation from $\RR$ to $\wt{\RR}$ is done in order to avoid the singularity at $\lambda=0$ that would have otherwise appeared after dividing \eqref{g0-sol} by $\Delta$.
  In particular, although this singularity ends up being of $\bigohnoscale{\lambda^{-1/3}}$, which is integrable (see remark~\ref{int-sing-r} below), it is more convenient to avoid it anyway in order to easily show via analyticity that the second half of~\eqref{g0-sol} does not contribute in~\eqref{utm-u2} and~\eqref{utm-v2}.
\end{rmk}

Now, from definitions (\ref{eqn:defn.sigma}--\ref{eqn:defn.rho}) and figures~\ref{fig:sigma}--\ref{fig:rho}, notice that if $\Im(\lambda) \leq 0$ then
\begin{equation} \label{im-ineq}
  \Im(\sigma(\lambda)) \geq 0,
  \quad
  \Im(\rho(\lambda)) \leq 0,
  \quad
  \Im(\alpha\rho(\lambda)) \geq 0,
  \quad
  \Im(\alpha^2\rho(\lambda)) \leq 0.
\end{equation}
Therefore, assuming that $u \in \Lebesgue_t^\infty((0,\infty); \Lebesgue_x^\infty(0, \infty))$ and $u_x, v \in \Lebesgue_t^\infty((0,\infty); \Lebesgue_x^1(0, \infty))$, the fourth $\lambda$ integral in each of the expressions~\eqref{utm-u3} and~\eqref{utm-v3} vanishes.
Specifically, in the case of the linear Schr\"odinger representation~\eqref{utm-u3}, thanks to analyticity in the region lying below $\wt{\RR}$ (and relying upon the aforementioned deformation below $\lambda=0$),
\begin{equation} \label{u-deform}
  \int_{\wt{\RR}} \re^{\ri\sigma(\lambda)x}  \,
  \frac{\rho(\lambda) \, \what u(-\sigma(\lambda), t)}{\Delta(\lambda)}
  \, \D\lambda
  =
  \lim_{R\to\infty} \int_{\mathcal C_R^-} \re^{\ri\sigma(\lambda)x}  \,
  \frac{\rho(\lambda) \, \what u(-\sigma(\lambda), t)}{\Delta(\lambda)}
  \, \D\lambda,
\end{equation}
where
\begin{equation} \label{eqn:lSAiry.defnCRminus}
  \mathcal C_R^- := \left\{R\re^{\ri\theta}: -\pi < \theta < 0\right\}.
\end{equation}

\begin{rmk} \label{int-sing-r}
  The unique zero of $\Delta(\lambda)$ at $\lambda=0$ is an integrable singularity of the relevant integrals in~\eqref{utm-u3} and~\eqref{utm-v3} since
  \begin{align}
    \frac{\rho(\lambda) \, \what u_0(-\sigma(\lambda))}{\Delta(\lambda)}
    &=
    \frac{1}{\alpha-1} \int_0^\infty u_0(x)  \,
    \frac{\re^{\ri\sigma(\lambda)x}}{\rho(\lambda) + \alpha^2 \sigma(\lambda)} \D x,
    \\
    \frac{- \what v_0(\rho(\lambda)) + \what v_0(\alpha^2 \rho(\lambda))}{\Delta(\lambda)}
    &=
    \frac{1}{\alpha-1} \int_0^\infty v_0(x)  \, \frac{\re^{-\ri \alpha^2 \rho(\lambda) x} - \re^{-\ri \rho(\lambda) x}}{\rho(\lambda)\left[\rho(\lambda) + \alpha^2 \sigma(\lambda)\right]} \D x,
  \end{align}
  and near $\lambda=0$ we have the series expansions
  \begin{align*}
    \frac{\re^{\ri\sigma(\lambda)x}}{\rho(\lambda) + \alpha^2 \sigma(\lambda)}
    &=
    \lambda^{-\frac 13} - \alpha^2 \lambda^{-\frac 16} + \bigoh{1},
    \\
    \frac{\re^{-\ri \alpha^2 \rho(\lambda) x} - \re^{-\ri \rho(\lambda) x}}{\rho(\lambda) \left[\rho(\lambda) + \alpha^2 \sigma(\lambda)\right]}
    &=
    \ri\left(1 -\alpha^2 \right)x \lambda^{-\frac 13} - \sqrt 3 \, x \lambda^{-\frac 16} + \bigoh{1}.
  \end{align*}
  However, in order to simplify justification of deformation~\eqref{u-deform} and those similar to it below, we have chosen to avoid $\lambda = 0$ by deforming locally into the lower half of the complex $\lambda$ plane.
\end{rmk}

Integrating by parts under the half line Fourier transform, we have
\begin{equation*}
  \what u(-\sigma(\lambda), t) \equiv \int_0^\infty \re^{\ri\sigma(\lambda) y} u(y, t) \D y
  =
  -\frac{u(0, t)}{i\sigma(\lambda)} - \frac{1}{i\sigma(\lambda)} \int_0^\infty \re^{\ri\sigma(\lambda) y} u_y(y, t) \D y.
\end{equation*}
Thus, back to~\eqref{u-deform}, for any $R>1$ (so that $\abs{\sigma(R \re^{\ri\theta})}-\abs{\rho(R \re^{\ri\theta})} = R^{1/2} - R^{1/3} > 0$), thanks to the first inequality in~\eqref{im-ineq} we find
\begin{align}
  \abs{
    \left(\alpha-1\right)
    \int_{\mathcal C_R^-} \re^{\ri\sigma(\lambda)x}  \,
    \frac{\rho(\lambda) \, \what u(-\sigma(\lambda), t)}{\Delta(\lambda)}
    \D\lambda
  }
  \hspace{-15em}&\hspace{15em}
  \nn\\
  &\leq
\begin{multlined}[t]
  \abs{u(0, t)}
  \abs{
    \int_{\mathcal C_R^-} \re^{\ri\sigma(\lambda)x}  \,
    \frac{\D\lambda}{\sigma(\lambda)\left[\rho(\lambda) + \alpha^2 \sigma(\lambda)\right]}
  }
  \\
  \hspace{5em}
  +
  \abs{
    \int_{\mathcal C_R^-} \re^{\ri\sigma(\lambda)x}  \int_0^\infty \re^{\ri\sigma(\lambda) y} u_y(y, t) dy \,
    \frac{\D\lambda}{\sigma(\lambda)\left[\rho(\lambda) + \alpha^2 \sigma(\lambda)\right]}
  }
\end{multlined}
  \nn\\
  &\leq
\begin{multlined}[t]
  \norm{u(t)}_{\Lebesgue_x^\infty(0, \infty)}
  \int_{-\pi}^0 \re^{-x\Im(\sigma(R \re^{\ri\theta}))} \frac{R \D\theta}{R^{\frac 12}\abs{\rho(R \re^{\ri\theta}) + \alpha^2 \sigma(R \re^{\ri\theta})}}
  \nn\\
  \hspace{5em}
  +
  \int_0^\infty \abs{u_y(y, t)}
  \int_{-\pi}^0 \re^{-(x+y)\Im(\sigma(R \re^{\ri\theta}))} \frac{R \D\theta}{R^{\frac 12}\abs{\rho(R \re^{\ri\theta}) + \alpha^2 \sigma(R \re^{\ri\theta})}} \D y
\end{multlined}
  \nn\\
  &\leq
  \left(\norm{u(t)}_{\Lebesgue_x^\infty(0, \infty)} + \norm{u_x(t)}_{\Lebesgue_x^1(0, \infty)}\right)
  \int_{-\pi}^0 \re^{x R^{\frac 12} \sin\left(\frac \theta 2\right)} \frac{R^{\frac 12} \D\theta}{\abs{\sigma(R \re^{\ri\theta})}-\abs{\rho(R \re^{\ri\theta})}}
  \nn\\
  &\overset{\mathclap{\theta = -2\phi}}{=}\hspace{0.8em}
  \left(\norm{u(t)}_{\Lebesgue_x^\infty(0, \infty)} + \norm{u_x(t)}_{\Lebesgue_x^1(0, \infty)}\right)\int_0^{\frac \pi 2} \re^{-x R^{\frac 12} \sin(\phi)} \frac{2R^{\frac12} \D\phi}{R^{\frac 12} - R^{\frac 13}}
  \nn\\
  &\leq
  \left(\norm{u(t)}_{\Lebesgue_x^\infty(0, \infty)} + \norm{u_x(t)}_{\Lebesgue_x^1(0, \infty)}\right)\frac{2R^{\frac 12}}{R^{\frac 12} - R^{\frac 13}}
  \int_0^{\frac \pi 2} \re^{-x R^{\frac 12} \sin(\phi)} \D\phi
  \nn\\
  &\leq
  \left(\norm{u(t)}_{\Lebesgue_x^\infty(0, \infty)} + \norm{u_x(t)}_{\Lebesgue_x^1(0, \infty)}\right)\frac{2R^{\frac 12}}{R^{\frac 12} - R^{\frac 13}}
  \int_0^{\frac \pi 2} \re^{-x R^{\frac 12} \cdot \frac 2\pi \phi} \D\phi
  \nn\\
  &=
  \left(\norm{u(t)}_{L_x^\infty(0, \infty)} + \norm{u_x(t)}_{L_x^1(0, \infty)}\right)\frac{2R^{\frac 12}}{R^{\frac 12} - R^{\frac 13}}
  \cdot
  \frac{\pi}{2xR^{\frac 12}} \, \Big(1-\re^{-xR^{\frac 12}}\Big),
\end{align}
where we have also used the convexity inequality $\sin(\phi) \geq 2\phi/\pi$, $0\leq \phi \leq \pi/2$, in the penultimate step.
Therefore,
\begin{equation} \label{u=0}
  \int_{\wt{\mathbb R}} \re^{\ri\sigma(\lambda)x}  \,
  \frac{\rho(\lambda) \, \what u(-\sigma(\lambda), t)}{\Delta(\lambda)}
  \D\lambda = 0, \quad x > 0, \ t\in\mathbb R.
\end{equation}
Similarly, by analyticity and exponential decay, it follows that
\begin{equation} \label{v=0}
  \int_{\wt{\RR}} \re^{\ri\sigma(\lambda)x} \, \frac{\what v(\alpha^j \rho(\lambda), t)}{\Delta(\lambda)}
  \D\lambda
  =
  \lim_{R\to\infty} \int_{\mathcal C_R^-} \re^{\ri\sigma(\lambda)x} \, \frac{\what v(\alpha^j\rho(\lambda), t)}{\Delta(\lambda)}
  \D\lambda
  = 0, \quad x>0, \ t\in\mathbb R, \ j=0, 2,
\end{equation}
since thanks to the first, second and fourth inequality in \eqref{im-ineq} we have the bound
\begin{align}
  \abs{
    \left(\alpha-1\right)
    \int_{\mathcal C_R^-} e^{i\sigma(\lambda)x} \, \frac{\what v(\alpha^j\rho(\lambda), t)}{\Delta(\lambda)}
    \D\lambda
  }
  &=
  \abs{
    \int_{\mathcal C_R^-} \re^{\ri\sigma(\lambda)x}  \int_0^\infty \re^{-\ri\alpha^j\rho(\lambda) y} v(y, t) \D y \,
    \frac{\D\lambda}{\rho(\lambda) \left[\rho(\lambda) + \alpha^2 \sigma(\lambda)\right]}
  }
  \nn\\
  &\leq
  \norm{v(t)}_{\Lebesgue_x^1(0, \infty)}
  \int_{-\pi}^0 \re^{-x \Im(\sigma(R \re^{\ri\theta}))} \frac{R \D\theta}{R^{\frac 13}  (R^{\frac 12} - R^{\frac 13})}
  \nn\\
  &\leq
  \norm{v(t)}_{\Lebesgue_x^1(0, \infty)}
  \frac{2R^{\frac 23}}{R^{\frac 12} - R^{\frac 13}}
  \cdot
  \frac{\pi}{2xR^{\frac 12}} \, \Big(1-\re^{-xR^{\frac 12}}\Big),
\end{align}
valid for $R>1$, with $j\in\{0, 2\}$.
Proceeding entirely analogously with representation~\eqref{utm-v3} for the Airy part of the problem, this time using all four inequalities in~\eqref{im-ineq}, we have for $R>1$
\begin{align}
  \abs{
    \left(\alpha-1\right)
    \int_{\mathcal C_R^-} \re^{\ri\alpha \rho(\lambda) x}
    \,
    \frac{\rho(\lambda) \, \what u(-\sigma(\lambda), t)}{\Delta(\lambda)}
    \D\lambda
  }
  \hspace{-12em}&\hspace{12em}
  \nn\\
  &\leq
  \left(\norm{u(t)}_{\Lebesgue_x^\infty(0, \infty)} + \norm{u_x(t)}_{\Lebesgue_x^1(0, \infty)}\right)
  \frac{R^{\frac 12}}{R^{\frac 12}-R^{\frac 13}}
  \int_{-\pi}^0 \re^{-x R^{\frac 13} \sin\left(\frac{\theta+2\pi}{3}\right)} \D\theta
  \nn\\
  &\overset{\mathclap{\theta+2\pi=3\phi}}{=}\hspace{1.1em}
  \left(\norm{u(t)}_{\Lebesgue_x^\infty(0, \infty)} + \norm{u_x(t)}_{\Lebesgue_x^1(0, \infty)}\right)
  \frac{6R^{\frac 12}}{R^{\frac 12}-R^{\frac 13}}
  \int_{\frac \pi 3}^{\frac \pi 2} \re^{-x R^{\frac 13} \sin(\phi)} \D\phi
  \nn\\
  &\leq
  \left(\norm{u(t)}_{\Lebesgue_x^\infty(0, \infty)} + \norm{u_x(t)}_{\Lebesgue_x^1(0, \infty)}\right)
  \frac{6R^{\frac 12}}{R^{\frac 12}-R^{\frac 13}}
  \int_{\frac \pi 3}^{\frac \pi 2} \re^{-x R^{\frac 13} \cdot \frac 2\pi \phi} \D\phi
  \nn\\
  &=
  \left(\norm{u(t)}_{\Lebesgue_x^\infty(0, \infty)} + \norm{u_x(t)}_{\Lebesgue_x^1(0, \infty)}\right)
  \frac{6R^{\frac 12}}{R^{\frac 12}-R^{\frac 13}}
  \cdot
  \frac{\pi}{2x R^{\frac 13}}
  \, \Big(\re^{-\frac 23 x R^{\frac 13}} - \re^{-x R^{\frac 13}} \Big)
\end{align}
and, for $R>1$ and $j\in\{0,2\}$,
\begin{multline*}
  \abs{
    \left(\alpha-1\right)
    \int_{\mathcal C_R^-} \re^{i\alpha \rho(\lambda) x}
    \,
    \frac{\what v(\alpha^j \rho(\lambda), t)}{\Delta(\lambda)}
    \D\lambda
  }
  \leq
  \norm{v(t)}_{\Lebesgue_x^1(0, \infty)} \frac{R^{\frac 23}}{R^{\frac 12} - R^{\frac 13}}
  \int_{-\pi}^0 \re^{-x R^{\frac 13} \sin\left(\frac{\theta+2\pi}{3}\right)} \, \D\theta
  \\
  \leq
  \norm{v(t)}_{\Lebesgue_x^1(0, \infty)} \frac{6R^{\frac 23}}{R^{\frac 12} - R^{\frac 13}}
  \cdot
  \frac{\pi}{2x R^{\frac 13}}
  \, \Big(e^{-\frac 23 x R^{\frac 13}} - \re^{-x R^{\frac 13}} \Big).
\end{multline*}
Hence, since both of the above bounds tend to zero as $R\to\infty$, we infer
\begin{equation} \label{v=02}
  \int_{\wt{\RR}} \re^{\ri\alpha \rho(\lambda) x}
  \,
  \frac{\left(1-\alpha^2\right)\rho(\lambda) \, \what u(-\sigma(\lambda), t) - \what v(\rho(\lambda), t) + \what v(\alpha^2 \rho(\lambda), t)}{\Delta(\lambda)}
  \D\lambda = 0, \quad x>0, \ t\in\mathbb R.
\end{equation}

In view of~\eqref{u=0},~\eqref{v=0} and~\eqref{v=02}, the representations~\eqref{utm-u3} and~\eqref{utm-v3} simplify to the \emph{explicit solution formulae}
\begin{multline} \label{utm-u-sol}
  u(x, t)
  =
  \frac{1}{2\pi} \int_{\RR} \re^{\ri kx-\ri k^2 t} \, \what u_0(k) \D k
  -
  \frac{1}{2\pi} \int_{\CSchro} \re^{\ri kx-\ri k^2 t} \, \what u_0(-k) \D k
  \\
  +
  \frac{1}{2\pi} \int_{\wt{\RR}} \re^{\ri\sigma(k)x+\ri k t}
  \,
  \frac{\left(1-\alpha^2\right)\rho(k) \, \what u_0(-\sigma(k)) - \what v_0(\rho(k)) + \what v_0(\alpha^2 \rho(k))}{\left(\alpha-1\right)\rho(k) \left[\rho(k) + \alpha^2 \sigma(k)\right]}
  \D k
\end{multline}
and
\begin{multline} \label{utm-v-sol}
  v(x, t)
  =
  \frac{1}{2\pi} \int_{\RR} \re^{\ri kx+\ri k^3 t} \, \what v_0(k) \D k
  +
  \frac{1}{2\pi} \int_{\CAiry} \re^{\ri kx+\ri k^3 t} \left[\alpha \, \what v_0(\alpha k) + \alpha^2 \, \what v_0(\alpha^2 k)\right] \D k
  \\
  +
  \frac{1}{2\pi} \int_{\wt{\RR}} \re^{\ri\alpha \rho(k) x+\ri k t}
  \,
  \frac{\left(1-\alpha^2\right)\rho(k) \, \what u_0(-\sigma(k)) - \what v_0(\rho(k)) + \what v_0(\alpha^2 \rho(k))}{\left(\alpha-1\right)\rho(k) \left[\rho(k) + \alpha^2 \sigma(k)\right]}
  \D k.
\end{multline}
The solution formulae in proposition~\ref{prop:lS-Airy} follow by
\begin{equation} \label{eqn:w-to-u-change-of-variable}
  w(x, t)=u(-x, t),
  \qquad\mbox{hence}\qquad
  \what w_0(k) = \what u_0(-k).
\end{equation}

\section{Linear Schr\texorpdfstring{\"o}{o}dinger and linearized Korteweg de Vries} \label{sec:lSlKdV}

We will prove proposition~\ref{prop:lS-lKdV} by solving the interface problem~\eqref{eqn:LS},~\eqref{eqn:lKdV},~\eqref{eqn:IfC.1}
where, as before, $g_j$ are merely a notational convenience but both $w_0$ and $v_0$ are given as initial data.
The introduction of the nonzero real parameter $a$ complicates this problem but, as we shall show, this does not preclude adaptation of the method described in~\S\ref{sec:lSAiry}.


As before, we reexpress the linear Schr\"odinger problem on the positive half line as $u(x, t) := w(-x, t)$ so that it satisfies problem~\eqref{eqn:LSAiryPositive}, has solution~\eqref{utm-u} for $g_0$ remaining to be determined, and obeys the global relation~\eqref{gr-ls-0}.

The linearized Korteweg de Vries (lKdV) problem has solution satisfying not~\eqref{utm-v} but
\begin{multline} \label{utm-v-lKdV}
  2\pi v(x, t)
  =
  \int_{\mathbb R} \re^{\ri kx+\ri (k^3-ak) t} \, \what v_0(k) \D k
  \\
  +
  \int_{\M{\mathcal C}3R} \re^{\ri \nu(k)x+\ri k^3 t} \frac{(\nu(\alpha k)-\nu(k))\what v_0(\nu(\alpha^2k)) - (\nu(\alpha^2k)-\nu(k))\what v_0(\nu(\alpha k))}{\nu(\alpha^2k)-\nu(\alpha k)} \nu'(k) \D k
  \\
  -
  \int_{\M{\mathcal C}3R} \re^{\ri \nu(k)x+\ri k^3 t} \left[ \left( \nu(k) - \nu(\alpha k) - \nu(\alpha^2 k) \right) \nu(k) + \nu(\alpha k) \nu(\alpha^2k) \right] \wt g_0(-\ri k^3, t) \nu'(k) \D k,
\end{multline}
where $\alpha = \re^{\ri\frac{2\pi}{3}}$ and $\M{\mathcal C}3R = \partial \M D3R$ for the region
\begin{equation}
  \M D3R := \left\{ k \in \mathbb C: \frac \pi 3 < \arg(k) < \frac{2\pi}{3} \mbox{ and } \abs k > R \right\},
\end{equation}
in which $R>0$ is sufficently large and $\nu$ is a biholomorphism from $\CC\setminus B(0,R)$ to the exterior of a suitable compact region satisfying
\begin{equation} \label{eqn:nu}
  \nu(k)^3-a\nu(k)=k^3
  \quad\mbox{with asymptotics}\quad
  \nu(k) = k+\bigoh{k^{-1}},
  \quad\mbox{and}\quad
  \nu'(k) = 1+\bigoh{k^{-2}},
\end{equation}
uniformly in $\arg(k)$ as $k\to\infty$.
That such $\nu$ exists is justified by~\cite[proposition~1]{ABS2022a} and~\cite[lemma~1]{Smi2026a}.
Solution formula~\eqref{utm-v-lKdV} may be obtained from~\cite[proposition~1.2 and equation~(1.54)]{Fok2008a} via a compact contour deformation and change of integration variable $k\mapsto\nu(k)$.
Similarly obtained is the global relation
\begin{equation}
  \re^{-\ri k^3 t} \, \what v(\nu(k), t)
  =
  \what v_0(\nu(k)) + \wt g_2(-\ri k^3, t) + \ri \nu(k) \, \wt g_1(-\ri k^3, t) - [\nu(k)^2-a] \wt g_0(-\ri k^3, t),
  \label{gr-lKdV-0}
\end{equation}
valid for all $k$ having $\Im(k)\leq0$ and $\abs k>R$.

\begin{rmk}
  The biholomorphism $\nu$ is not analogous to the root functions $\sigma,\rho$ introduced in~\S\ref{sec:lSAiry}.
  This is evident from the stated large $\la$ asymptotic for $\nu$: it is approximately the identity function.
  The purpose of $\nu$ is that, when composed with the polynomial $k\mapsto k^3-ak$, which is the symbol of the spatial differential operator associated with the PDE of~\eqref{eqn:lKdV}, the monomial $k\mapsto k^3$ is produced.
  To analyse problem~\eqref{eqn:LS},~\eqref{eqn:lKdV},~\eqref{eqn:IfC.1} in full, we shall have to use both $\sigma,\rho$ as before, but also $\nu$.

  The precise formulation of equation~\eqref{utm-v-lKdV} is a change of variables $k\mapsto \nu(k)$ away from the traditional formulation as it appears in~\cite{Fok2008a}.
  We find our version more convenient, because it has simple contours $\M{\mathcal C}3R$ composed of only rays and circular arcs, instead of more complicated curves satisfying $\Im(k^3-ak)=0$.
\end{rmk}

We use the same root functions $\sigma$ and $\rho$ as in~\S\ref{sec:lSAiry} to obtain identities~\eqref{gr-ls-1},~\eqref{gr-ls-2} and, in place of the corresponding Airy identities, the new lKdV identities
\begin{align}
  \notag
  \re^{-\ri\lambda t} \, \what v(\nu(\rho(\lambda)), t)
  &=
  \what v_0(\nu(\rho(\lambda))) + \wt g_2(-\ri \lambda, t) + \ri \nu(\rho(\lambda)) \, \wt g_1(-\ri\lambda, t) - [\nu(\rho(\lambda))^2-a] \wt g_0(-\ri\lambda, t),
  \\
  &\hspace{10em}\tfrac{-3\pi}2\leq\arg(\lambda)\leq 0, \qquad \abs\la \mbox{ sufficiently large},
  \label{gr-lkdv-1}
  \\
  \notag
  \re^{-\ri\lambda t} \, \what v(\alpha^2 \rho(\lambda), t)
  &=
  \what v_0(\nu(\alpha^2 \rho(\lambda))) + \wt g_2(-\ri \lambda, t) + \ri \nu(\alpha^2 \rho(\lambda)) \, \wt g_1(-\ri\lambda, t) - [\nu(\alpha^2 \rho(\lambda))^2-a] \wt g_0(-\ri\lambda, t),
  \\
  &\hspace{10em}-\pi\leq\arg(\lambda)\leq \tfrac\pi2, \qquad \abs\la \mbox{ sufficiently large},
  \label{gr-lkdv-2}
\end{align}
where the relevant regions of validity are due to the transformations of the region
\[
  \{k\in\CC:\text{Im}(k) \leq 0 \mbox{ and } \abs k > R\}
\]
under the prescribed changes of variable, as may be deduced from figures~\ref{fig:sigma} and~\ref{fig:rho}.

We combine identities~\eqref{gr-ls-2},~\eqref{gr-lkdv-1}, and~\eqref{gr-lkdv-2} to form a $3\times3$ linear system for the transformed interface values $\wt g_0(-\ri\lambda, t)$, $\wt g_1(-\ri\lambda, t)$,  $\wt g_2(-\ri\lambda, t)$:
\begin{equation} \label{lin-sys-lkdv}
  \begin{pmatrix}
    \sigma(\lambda) & 1 & 0 \\
    \nu(\rho(\lambda)^2)-a & \nu(\rho(\lambda)) & 1 \\
    \nu(\alpha^2 \rho(\lambda))^2-a & \nu(\alpha^2 \rho(\lambda)) & 1
  \end{pmatrix}
  \begin{pmatrix}
    -\wt g_0(-\ri\lambda, t) \\
    \ri\, \wt g_1(-\ri\lambda, t) \\
    \wt g_2(-\ri\lambda, t)
  \end{pmatrix}
  =
  \re^{-\ri\lambda t}
  \begin{pmatrix}
    \what u(-\sigma(\lambda), t) \\
    \what v(\nu(\rho(\lambda)), t) \\
    \what v(\nu(\alpha^2 \rho(\lambda)), t)
  \end{pmatrix}
  -
  \begin{pmatrix}
    \what u_0(-\sigma(\lambda)) \\
    \what v_0(\nu(\rho(\lambda))) \\
    \what v_0(\nu(\alpha^2 \rho(\lambda)))
  \end{pmatrix},
\end{equation}
for $\Im(\la)\leq0$ with $\abs\la$ large enough.
The determinant of this system is
\begin{equation} \label{eqn:lSlKdV.Delta}
  \Delta(\lambda) = \left( \nu(\alpha^2\rho(\la))-\nu(\rho(\la)) \right) \left( \nu(\alpha^2\rho(\la))+\nu(\rho(\la))-\sigma(\la) \right),
\end{equation}
which, by the properties of $\nu$ and the analysis following equation~\eqref{eqn:lSAiry.Delta}, has no zeros outside a bounded region.
The first entry of the solution of linear system~\eqref{lin-sys-lkdv} satisfies
\begin{multline} \label{g0-sol-lKdV}
  \Delta(\lambda) \, \wt g_0(-\ri\lambda, t)
  =
  \left[\left( \nu(\rho(\lambda))-\nu(\alpha^2\rho(\lambda)) \right) \what u_0(-\sigma(\lambda)) - \what v_0(\nu(\rho(\lambda))) + \what v_0(\nu(\alpha^2\rho(\lambda)))\right]
  \\
  - \re^{-\ri\lambda t} \left[\left( \nu(\rho(\lambda))-\nu(\alpha^2\rho(\lambda)) \right) \what u(-\sigma(\lambda), t) - \what v(\nu(\rho(\lambda)), t) + \what v(\nu(\alpha^2\rho(\lambda)), t)\right],
\end{multline}
for $\Im(\la)\leq0$ with $\abs\la$ large enough.

We select the same changes of variables in the latter integrals of solution representations~\eqref{utm-u} and \eqref{utm-v-lKdV}, arriving at
\begin{equation} \label{utm-u2-lS}
  2\pi u(x, t)
  =
  \int_{\mathbb R} \re^{\ri kx-\ri k^2 t} \, \what u_0(k) \D k
  -
  \int_{\CSchro} \re^{\ri kx-\ri k^2 t} \, \what u_0(-k) \D k
  +
  \int_{\wt\RR} \re^{\ri\sigma(\lambda)x+\ri\lambda t}  \, \wt g_0(-\ri\lambda, t) \D\lambda.
\end{equation}
and
\begin{multline} \label{utm-v2-lKdV}
  2\pi v(x, t)
  =
  \int_{\mathbb R} \re^{\ri kx+\ri (k^3-ak) t} \, \what v_0(k) \D k
  \\
  +
  \int_{\M{\mathcal C}3R} \re^{\ri \nu(k)x+\ri k^3 t} \frac{(\nu(\alpha k)-\nu(k))\what v_0(\nu(\alpha^2k)) - (\nu(\alpha^2k)-\nu(k))\what v_0(\nu(\alpha k))}{\nu(\alpha^2k)-\nu(\alpha k)} \nu'(k) \D k
  \\
  -
  \int_{\wt\RR} \re^{\ri\nu(\alpha \rho(\lambda)) x+\ri\lambda t} N(\la) \wt g_0(-\ri\lambda, t) \D\lambda,
\end{multline}
in which the contour $\wt{\RR}$ bypasses $\lambda=0$ via a semicircle of sufficiently large radius lying in the lower half of the complex $\lambda$-plane and
\[
  N(\la) = \frac{\left( \nu(\alpha \rho(\la)) - \nu(\alpha^2\rho(\la)) - \nu(\rho(\la)) \right) \nu(\alpha\rho(\la)) + \nu(\alpha^2\rho(\la)) \nu(\rho(\la))}{3\alpha^2\rho(\la)^2} \nu'(\alpha\rho(\la)).
\]
The criteria on $\nu$ guarantee that $N(\la) = 1 + \bigoh{\la^{-2/3}}$, uniformly in the argument of $\la$, as $\la\to\infty$ within $\clos\CC^-$, and $N$ is holomorphic to the right of the contour $\wt\RR$ and continuous onto $\wt\RR$.
Note that the semicircular component of $\wt{\RR}$ must now be of radius at least $R^3$, and possibly larger still to avoid zeros of $\Delta$.
This is in contrast to the situation in~\S\ref{sec:lSAiry} wherein the radius need only have been positive.

Then, using~\eqref{g0-sol-lKdV}, we substitute for $\wt g_0(-\ri\lambda, t)$ to obtain
\begin{multline} \label{utm-u3-lS}
  2\pi u(x, t)
  =
  \int_{\RR} \re^{\ri kx-\ri k^2 t} \, \what u_0(k) \D k
  -
  \int_{\CSchro} \re^{\ri kx-\ri k^2 t} \, \what u_0(-k) \D k
  \\
  +
  \int_{\wt{\RR}} \re^{\ri\sigma(\lambda)x+\ri\lambda t}
  \,
  \frac{\left( \nu(\rho(\lambda))-\nu(\alpha^2\rho(\lambda)) \right) \what u_0(-\sigma(\lambda)) - \what v_0(\nu(\rho(\lambda))) + \what v_0(\nu(\alpha^2\rho(\lambda)))}{\Delta(\lambda)}
  \D \lambda
\end{multline}
and
\begin{multline} \label{utm-v3-lKdV}
  2\pi v(x, t)
  =
  \int_{\mathbb R} \re^{\ri kx+\ri (k^3-ak) t} \, \what v_0(k) \D k
  \\
  +
  \int_{\M{\mathcal C}3R} \re^{\ri \nu(k)x+\ri k^3 t} \frac{(\nu(\alpha k)-\nu(k))\what v_0(\nu(\alpha^2k)) - (\nu(\alpha^2k)-\nu(k))\what v_0(\nu(\alpha k))}{\nu(\alpha^2k)-\nu(\alpha k)} \nu'(k) \D k
  \\
  +
  \int_{\wt\RR} \re^{\ri\nu(\alpha \rho(\lambda)) x+\ri\lambda t} N(\la)
  \,
  \frac{\left( \nu(\rho(\lambda))-\nu(\alpha^2\rho(\lambda)) \right) \what u_0(-\sigma(\lambda)) - \what v_0(\nu(\rho(\lambda))) + \what v_0(\nu(\alpha^2\rho(\lambda)))}{\Delta(\lambda)}
  \D\lambda
  ,
\end{multline}
where the removal of the integrals involving $\what u(\argdot,t),\what v(\argdot,t)$ is justified by the same argument as that in~\S\ref{sec:lSAiry}, via the asymptotics on $N$ given above and $\nu$ described in its definition~\eqref{eqn:nu}.
Proposition~\ref{prop:lS-lKdV} follows from equations~(\ref{utm-u3-lS}--\ref{utm-v3-lKdV}) via change of variables~\eqref{eqn:w-to-u-change-of-variable}.

\section{Heat and Airy} \label{sec:heatAiry}

We solve the interface problem~\eqref{eqn:h},~\eqref{eqn:Airy},~\eqref{eqn:IfC.1} for initial data $w_0$ and $v_0$, and thereby prove proposition~\ref{prop:lS-lKdV}.

The Airy problem has solution given by~\eqref{utm-v} and global relation~\eqref{gr-airy-0}.
We rewrite the heat problem using $u(x,t) = w(-x,t)$ as
\begin{subequations} \label{eqn:heatAiryPositive}
\begin{align}
  \label{eqn:heatAiryPositive.hPDE} \tag{\theparentequation.PDE}
  u_t - u_{xx} &= 0 & (x,t) &\in (0,\infty)\times(0,T), \\
  \label{eqn:heatAiryPositive.hInC} \tag{\theparentequation.InC}
  u(x,0) &= u_0(x) := w_0(-x) & x &\in(0,\infty), \\
  \label{eqn:heatAiryPositive.IfC} \tag{\theparentequation.IfC}
  \partial_x^j u(0,t) &= (-1)^j g_j(t) & t &\in [0,T], \qquad j\in\{0,1\},
\end{align}
\end{subequations}
From~\cite[equation~(3.9)]{DTV2014a}, the $u$ problem has solution
\begin{equation} \label{eqn:HeatAiry.HeatEF}
  2\pi u(x,t) = \int_\RR \re^{\ri kx-k^2t}\hat u_0(k)\D k - \int_{\CHeat} \re^{\ri kx-k^2t} \hat u_0(-k) \D k - \int_{\CHeat} \re^{\ri kx - k^2t} 2\ri k \tilde g_0(k^2,t)\D k,
\end{equation}
in which $\CHeat$ denotes the positively oriented boundary of the region
\begin{equation}
  \DHeat := \left\{ k \in \mathbb C: \frac \pi 4 < \arg(k) < \frac{3\pi}{4} \right\}.
\end{equation}
The global relation
\begin{equation} \label{eqn:HeatAiry.HeatGR}
  \hat u_0(k) - \left[ -\tilde g_1(k^2,t) + \ri k \tilde g_0(k^2,t) \right] = \re^{k^2t} \hat u(k,t), \quad \Im(k)\leq0.
\end{equation}
holds for all $t\geq0$.

Analogously to the problem studied in~\S\ref{sec:lSAiry}, we seek complex functions $\tau$ and $\rho$ having branch cuts, of square root and cube root type respectively, beginning at $0$ and extending to $\infty$ along the positive imaginary axis, such that
\[
  \la=\ri[\pm\tau(\la)]^2,
  \qquad
  \la=[\alpha^j\rho(\la)]^3,
  \qquad j\in\{0,1,2\},
\]
where $\alpha=\re^{\ri\frac{2\pi}3}$ is a primitive cube root of unity.
We find it convenient to define $\rho$ the same way as in~\S\ref{sec:lSAiry}, by equation~\eqref{eqn:defn.rho}.
We define $\tau$ by equation~\eqref{eqn:defn.tau}; see figure~\ref{fig:tau}.

\begin{figure}
  \centering
  \includegraphics[scale=0.7]{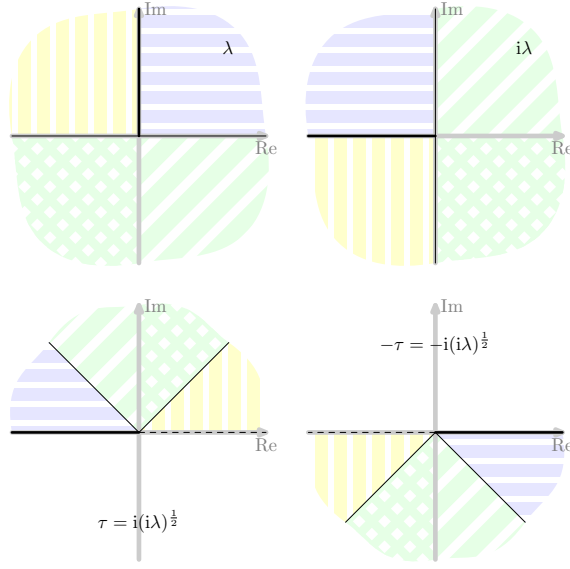}
  \caption{The locus of $\pm \tau(\lambda)$.}
  \label{fig:tau}
\end{figure}

Applying the maps $k=\rho(\la)$ and $k=\alpha^2\rho(\la)$ to equation~\eqref{gr-airy-0}, we obtain~\eqref{gr-airy-1} and~\eqref{gr-airy-2}, and applying the map $k=-\tau(\la)$ to equation~\eqref{eqn:HeatAiry.HeatGR} we arrive at
\begin{equation}
  \re^{-\ri\la} \hat u(-\tau(\la),t) = \hat u_0(-\tau(\la)) + \tilde g_1(-\ri\la,t) + \ri\tau(\la) \tilde g_0(-\ri\la,t),
\end{equation}
which is valid for all $\la\in\CC$ so, in particular, is valid for $\la\in\clos(\CC^-)$.
Combining these into a linear system, we have
\begin{equation}
  \begin{pmatrix}
    -\ri\tau(\la) & -\ri & 0 \\
    \rho(\la)^2 & \rho(\la) & 1 \\
    \alpha\rho(\la)^2 & \alpha^2\rho(\la) & 1
  \end{pmatrix}
  \begin{pmatrix}
    -\tilde g_0(-\ri\la,t) \\
    \ri \tilde g_1(-\ri\la,t) \\
    \tilde g_2(-\ri\la,t)
  \end{pmatrix}
  =
  \re^{-\ri\la t}
  \begin{pmatrix}
    \hat u(-\tau(\la),t) \\
    \hat v(\rho(\la),t) \\
    \hat v(\alpha^2\rho(\la),t)
  \end{pmatrix}
  -
  \begin{pmatrix}
    \hat u_0(-\tau(\la)) \\
    \hat v_0(\rho(\la)) \\
    \hat v_0(\alpha^2\rho(\la))
  \end{pmatrix},
\end{equation}
which holds on $\clos(\CC^-)$.
Its determinant $\Delta$ is given by
\begin{equation} \label{eqn:heatAiry.Delta}
  \Delta(\lambda) = -\ri \left(\alpha-1\right) \rho(\lambda) \left[\rho(\lambda) + \alpha^2 \tau(\lambda)\right],
\end{equation}
which has no nonzero zeros,
and the first entry of its solution satisfies
\begin{multline} \label{eqn:HeatAiry.g0}
  \Delta(\la)\tilde g_0(-\ri\la,t) = \left[ (1-\alpha^2)\rho(\la) \hat u_0(-\tau(\la)) + \ri \hat v_0(\rho(\la)) - \ri \hat v_0(\alpha^2\rho(\la)) \right] \\
  - \re^{-\ri\la t} \left[ (1-\alpha^2) \rho(\la) \hat u (-\tau(\la),t) + \ri \hat v(\rho(\la),t) - \ri \hat v(\alpha^2\rho(\la),t) \right]
\end{multline}
on $\clos(\CC^-)$.

In order to prepare the latter integrals of equations~\eqref{eqn:HeatAiry.HeatEF} and~\eqref{utm-v} for substitution of $\tilde g_0(-\ri\la,t)$ from equation~\eqref{eqn:HeatAiry.g0}, we follow the same approach as in previous sections: select the change of variables which maps $\la\in\RR$ to $k$ in the relevant contour.
Indeed, for the latter integral of equation~\eqref{utm-v}, we use the same change of variables as before, $k=\alpha\rho(\la)$, because it maps $\la\in\RR$ to $k\in \CAiry$, and for the latter integral of equation~\eqref{eqn:HeatAiry.HeatEF}, we use $k=\tau(\la)$.
Anticipating the zero of $\Delta$ at $0$, we make the usual contour deformation from $\RR$ to $\wt\RR$, then make the substitutions for $\tilde g_0(-\ri\la,t)$, obtaining
\begin{multline} \label{utm-u3-heat}
  2\pi u(x, t)
  =
  \int_{\RR} \re^{\ri kx-k^2 t} \, \what u_0(k) \D k
  -
  \int_{\CHeat} \re^{\ri kx-k^2 t} \, \what u_0(-k) \D k
  \\
  +
  \int_{\wt{\RR}} \re^{\ri\tau(\lambda)x+\ri\lambda t}
  \,
  \frac{\left(1-\alpha^2\right)\rho(\lambda) \, \what u_0(-\tau(\lambda)) + \ri \what v_0(\rho(\lambda)) - \ri \what v_0(\alpha^2 \rho(\lambda))}{\Delta(\lambda)}
  \D \lambda
  \\
  -
  \int_{\wt{\RR}} \re^{\ri\tau(\lambda)x}  \,
  \frac{\left(1-\alpha^2\right)\rho(\lambda) \, \what u(-\tau(\lambda), t) + \ri \what v(\rho(\lambda), t) - \ri \what v(\alpha^2 \rho(\lambda), t)}{\Delta(\lambda)}
  \D\lambda
\end{multline}
and
\begin{multline} \label{utm-v3-lKdV-with-heat}
  2\pi v(x, t)
  =
  \int_{\RR} \re^{\ri kx+\ri k^3 t} \, \what v_0(k)\D k
  +
  \int_{\CAiry} \re^{\ri kx+\ri k^3 t} \left[\alpha \, \what v_0(\alpha k) + \alpha^2 \, \what v_0(\alpha^2 k)\right] \D k
  \\
  +
  \int_{\wt{\RR}} \re^{\ri\alpha \rho(\lambda) x+\ri\lambda t}
  \,
  \frac{\left(1-\alpha^2\right)\rho(\lambda) \, \what u_0(-\tau(\lambda)) + \ri \what v_0(\rho(\lambda)) - \ri \what v_0(\alpha^2 \rho(\lambda))}{\Delta(\lambda)}
  \D\lambda
  \\
  -
  \int_{\wt{\RR}} \re^{\ri\alpha \rho(\lambda) x}
  \,
  \frac{\left(1-\alpha^2\right)\rho(\lambda) \, \what u(-\tau(\lambda), t) + \ri \what v(\rho(\lambda), t) - \ri \what v(\alpha^2 \rho(\lambda), t)}{\Delta(\lambda)}\D\lambda.
\end{multline}

The latter three of inequalities~\eqref{im-ineq} still hold for $\Im(\la)\leq0$, but the first is replaced by the stronger
\begin{equation}
  \Im(\tau(\la)) \geq \frac1{\sqrt2}\abs\la \geq 0.
\end{equation}
This means that, for $x>0$, we may split the exponential
\[
  \re^{\ri\tau(\la)x} = \re^{\ri\tau(\la)\frac x2} \bigoh{\re^{\frac{-x\abs\la}{2\sqrt2}}},
\]
which greatly simplifies the justification of
\begin{align*}
  \int_{\wt{\RR}} \re^{\ri\tau(\lambda)x}  \,
  \frac{\rho(\lambda) \, \what u(-\tau(\lambda), t)}{\Delta(\lambda)}
  \D\lambda &= 0, & &x > 0, \ t\geq0, \\
  \int_{\wt{\RR}} \re^{\ri\tau(\lambda)x}  \,
  \frac{\rho(\lambda) \, \what v(\alpha^j\rho(\lambda), t)}{\Delta(\lambda)}
  \D\lambda &= 0, & &x > 0, \ t\geq0, \ j = 0,2,
\end{align*}
compared with the argument for equations~\eqref{u=0} and~\eqref{v=0} in~\S\ref{sec:lSAiry}.
That
\[
  \int_{\wt{\RR}} \re^{\ri\alpha \rho(\lambda) x}
  \,
  \frac{\left(1-\alpha^2\right)\rho(\lambda) \, \what u(-\tau(\lambda), t) + \ri \what v(\rho(\lambda), t) - \ri \what v(\alpha^2 \rho(\lambda), t)}{\Delta(\lambda)}
  \D\lambda = 0, \quad x>0, \ t\in\mathbb R
\]
follows by the same reasoning as did equation~\eqref{v=02}.
In this way, we obtain explicit solution formulae
\begin{multline} \label{eqn:heatAiry.utm-u-sol}
  u(x, t)
  =
  \frac1{2\pi} \int_{\RR} \re^{\ri kx-k^2 t} \, \what u_0(k) \D k
  -
  \frac1{2\pi} \int_{\CHeat} \re^{\ri kx-k^2 t} \, \what u_0(-k) \D k
  \\
  +
  \frac1{2\pi} \int_{\wt{\RR}} \re^{\ri\tau(\lambda)x+\ri\lambda t}
  \,
  \frac{\left(1-\alpha^2\right)\rho(\lambda) \, \what u_0(-\tau(\lambda)) + \ri \what v_0(\rho(\lambda)) - \ri \what v_0(\alpha^2 \rho(\lambda))}{\Delta(\lambda)}
  \D \lambda,
\end{multline}
and
\begin{multline} \label{eqn:heatAiry.utm-v-sol}
  v(x, t)
  =
  \frac1{2\pi} \int_{\RR} \re^{\ri kx+\ri k^3 t} \, \what v_0(k)\D k
  +
  \frac1{2\pi} \int_{\CAiry} \re^{\ri kx+\ri k^3 t} \left[\alpha \, \what v_0(\alpha k) + \alpha^2 \, \what v_0(\alpha^2 k)\right] \D k
  \\
  +
  \frac1{2\pi} \int_{\wt{\RR}} \re^{\ri\alpha \rho(\lambda) x+\ri\lambda t}
  \,
  \frac{\left(1-\alpha^2\right)\rho(\lambda) \, \what u_0(-\tau(\lambda)) + \ri \what v_0(\rho(\lambda)) - \ri \what v_0(\alpha^2 \rho(\lambda))}{\Delta(\lambda)}
  \D\lambda.
\end{multline}
Proposition~\ref{prop:heat-Airy} follows from equations~(\ref{eqn:heatAiry.utm-u-sol}--\ref{eqn:heatAiry.utm-v-sol}) via change of variables~\eqref{eqn:w-to-u-change-of-variable}.

\section{Biharmonic Schr\texorpdfstring{\"o}{o}dinger and Airy} \label{sec:bihAiry}

We solve the problem~\eqref{eqn:biS},~\eqref{eqn:Airy},~\eqref{eqn:IfC.2}, which features a dispersive equation of higher order on the left half line, thereby proving proposition~\ref{prop:biS-Airy}.
The increased number of interface conditions in comparison to previous problems reflects the higher order of the biharmonic Schr\"odinger equation versus the Schr\"odinger or heat equations.
Indeed, when posed independently on a half line, the biharmonic Schr\"odinger equation requires one more boundary condition than do the linear Schr\"odinger or heat equations.

As before, we make the change of variable $x \mapsto -x$ in order to map the negative half line problem~\eqref{eqn:biS} to one on the positive half line for the function $u(x, t) := w(-x, t)$, i.e.
\begin{subequations} \label{eqn:biSAiryPositive}
\begin{align}
  \label{eqn:biSAiryPositive.biSPDE} \tag{\theparentequation.PDE}
  \ri u_t + u_{xxxx} &= 0 & (x,t) &\in (0,\infty)\times(0,T), \\
  \label{eqn:biSAiryPositive.biSInC} \tag{\theparentequation.InC}
  u(x,0) &= u_0(x) := w_0(-x) & x &\in(0,\infty), \\
  \label{eqn:biSAiryPositive.biSIfC} \tag{\theparentequation.IfC}
  \partial_x^j u(0,t) &= (-1)^j g_j(t) & t &\in [0,T], \qquad j\in\{0,1,2\}.
\end{align}
\end{subequations}
Hence, the interface problem can equivalently be described by~\eqref{eqn:Airy} and~\eqref{eqn:biSAiryPositive}.
As before, we work under the assumption of sufficiently smooth data that decay fast enough at infinity.

The solution of~\eqref{eqn:Airy} satisfies integral representation~\eqref{utm-v}.
Furthermore, in~\cite{OY2019a}, the following solution representation was derived for the linear biharmonic Schr\"odinger equation~\eqref{eqn:biSAiryPositive} by means of the unified transform method:
\begin{multline} \label{utm-bih}
  2\pi u(x, t)
  =
  \int_{\mathbb R} \re^{\ri kx+\ri k^4t} \, \what u_0(k) \D k
  -
  \int_{\CbiSone} e^{\ri kx+\ri k^4t} \left[\left(1+\ri\right) \what u_0(-\ri k) - \ri \, \what u_0(-k) \right] \D k
  \\
  -
  \int_{\CbiStwo} \re^{\ri kx+\ri k^4t} \left[\left(1-\ri\right) \what u_0(\ri k) + \ri \, \what u_0(-k) \right] \D k
  \\
  -
  \int_{\CbiSone} \re^{\ri kx+\ri k^4t} 2\left(1-\ri\right)k^2 \left[ k \, \widetilde g_0(-\ri k^4, t) -  \widetilde g_1(-\ri k^4, t) \right] \D k
  \\
  -
  \int_{\CbiStwo} \re^{\ri kx+\ri k^4t} 2 \left(1+\ri\right) k^2 \left[ k \, \widetilde g_0(-\ri k^4, t) + \widetilde g_1(-\ri k^4, t) \right] \D k,
\end{multline}
where a hat denotes the half line Fourier transform~\eqref{hlft-def}, a tilde denotes the temporal transform~\eqref{tilde-def}, and the complex contours $\CbiSone$ and $\CbiStwo$ are the positively oriented boundaries of the second and fourth octants.
Also provided in~\cite{OY2019a} is the linear biharmonic Schr\"odinger global relation, valid for $\Im(k) \leq 0$,
\begin{equation} \label{gr-bih-0}
  \re^{-\ri k^4 t} \, \what u(k, t)
  =
  \what u_0(k) - \ri \, \wt g_3(-\ri k^4, t) + k \, \wt g_2(-\ri k^4, t) + \ri k^2 \, \wt g_1(-\ri k^4, t) - k^3 \, \wt g_0(-\ri k^4, t).
\end{equation}

For the root type function $\gamma$ defined by~\eqref{eqn:defn.gamma}, the changes of variable $k = \ri^j \gamma(\lambda)$, $j  = 0, 1, 2, 3$, all imply that $\lambda = k^4$ with a branch cut along the positive imaginary axis of the complex $\lambda$ plane.
Then, the global relation~\eqref{gr-bih-0} yields four additional identities:
for $j\in\{0,1,2,3\}$,
\begin{subequations} \label{eqn:bihS-GR-j}
\begin{multline}
  \re^{-\ri \lambda t} \what u(\ri^j\gamma(\lambda), t)
  =
  \what u_0(\ri^j\gamma(\lambda)) - \ri \wt g_3(-\ri\lambda, t) + \ri^j\gamma(\lambda) \wt g_2(-\ri\lambda, t) + \ri^{2j}\ri\gamma(\lambda)^2 \wt g_1(-\ri\lambda, t)
  \\
  - \ri^{3j}\gamma(\lambda)^3 \wt g_0(-\ri\lambda, t),
  \tag{\theparentequation.$j$}
\end{multline}
\end{subequations}
valid for
\begin{equation} \label{eqn:bihS-GR-j-validity}
  \la=0 \mbox{ or }
  \begin{cases}
    -\frac{3\pi}2 \leq \arg(\la) \leq 0 & \mbox{if } j=0, \\
    \lambda\in\emptyset & \mbox{if } j=1, \\
    0 \leq \arg(\la) \leq \frac{\pi}2 & \mbox{if } j=2, \\
    \lambda\in\CC & \mbox{if } j=3.
  \end{cases}
\end{equation}
The domains of validity have been determined by tracking the region $\Im(k) \leq 0$ through each of the four transformations, as shown in figure~\ref{fig:gamma}.
Note that equation~(\ref{eqn:bihS-GR-j}.1) is only valid for $\lambda=0$ and so it is not useful, analogously to~\eqref{gr-airy-3} in the case of the Airy equation.
However equation~(\ref{eqn:bihS-GR-j}.3) is valid for all complex $\lambda$.

\begin{figure}[ht!]
  \centering
  \includegraphics[scale=0.7]{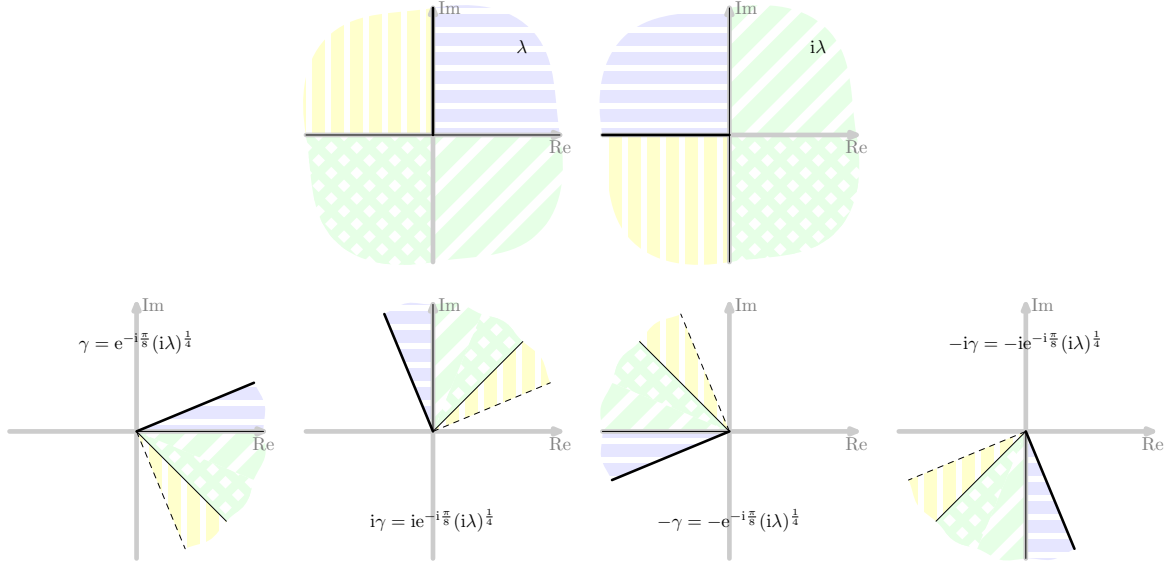}
  \caption{The sign of $\ri^j \gamma(\lambda)$, $j=0, 1, 2, 3$.}
  \label{fig:gamma}
\end{figure}

For our purposes, we will use identities~(\ref{eqn:bihS-GR-j}.0) and~(\ref{eqn:bihS-GR-j}.3) combined with the Airy identities~\eqref{gr-airy-1} and~\eqref{gr-airy-2}, since all four of them are valid for $\Im(\Delta(\la)) \leq 0$.
The resulting linear system will allow us to express the boundary value transforms $\wt g_0(-\ri\lambda, t), \wt g_1(-\ri\lambda, t), \wt g_2(-\ri\lambda, t), \wt g_3(-\ri\lambda, t)$ in terms of the half line Fourier transforms of the initial data, namely $\what u_0(\gamma(\lambda))$, $\what u_0(-\ri\gamma(\lambda))$, $\what v_0(\rho(\lambda))$, $\what v_0(\alpha^2 \rho(\lambda))$, and the corresponding transforms of the solutions $u$ and $v$.
As before, these latter transforms are unknown but will be shown to have zero contribution when the relevant $\lambda$ integrals are applied.
Specifically, our $4\times 4$ system reads, for $\Im\la\leq0$,
\begin{equation}
  \label{lin-sys-bih}
  \begin{pmatrix}
    \gamma(\lambda)^3 &\gamma(\lambda)^2 &\gamma(\lambda) &1
    \\
    \ri \gamma(\lambda)^3 &-\gamma(\lambda)^2 &-\ri \gamma(\lambda) &1
    \\
    \rho(\lambda)^2 &\rho(\lambda) &1 &0
    \\
    \alpha \rho(\lambda)^2 &\alpha^2 \rho(\lambda) &1 &0
  \end{pmatrix}
  \begin{pmatrix}
    -\wt g_0(-\ri\lambda, t)
    \\
    \ri\, \wt g_1(-\ri\lambda, t)
    \\
    \wt g_2(-\ri\lambda, t)
    \\
    -i \, \wt g_3(-\ri\lambda, t)
  \end{pmatrix}
  =
  \begin{pmatrix}
    \re^{-\ri\lambda t} \, \what u(\gamma(\lambda), t) - \what u_0(\gamma(\lambda))
    \\
    \re^{-\ri\lambda t} \, \what u(-\ri\gamma(\lambda), t) - \what u_0(-\ri\gamma(\lambda))
    \\
    \re^{-\ri\lambda t} \, \what v(\rho(\lambda), t) - \what v_0(\rho(\lambda))
    \\
    \re^{-\ri\lambda t} \, \what v(\alpha^2 \rho(\lambda), t) - \what v_0(\alpha^2 \rho(\lambda))
  \end{pmatrix},
\end{equation}
and yields the following expressions for the boundary value transforms $\wt g_0(-\ri\lambda, t)$ and $\wt g_1(-\ri\lambda, t)$:
\begin{align}
  \Delta(\lambda) \, \wt g_0(-\ri\lambda, t)
  &=
  \left(\alpha^2-1\right) \rho(\lambda)
  \left[\re^{-\ri\lambda t} \, \what u(\gamma(\lambda), t) - \what u_0(\gamma(\lambda))\right]
  \nn\\*
  &\qquad
  -
  \left(\alpha^2-1\right) \rho(\lambda) \left[\re^{-\ri\lambda t} \, \what u(-\ri\gamma(\lambda), t) - \what u_0(-\ri\gamma(\lambda))\right]
  \nn\\*
  &\qquad
  -
  \gamma(\lambda) \left[\alpha^2 \left(1+\ri\right) \rho(\lambda) - 2\gamma(\lambda)\right]
  \left[\re^{-\ri\lambda t} \, \what v(\rho(\lambda), t) - \what v_0(\rho(\lambda))
  \right]
  \nn\\*
  &\qquad
  +
  \gamma(\lambda) \left[\left(1+\ri\right) \rho(\lambda) - 2\gamma(\lambda)\right]
  \left[\re^{-\ri\lambda t} \, \what v(\alpha^2 \rho(\lambda), t) - \what v_0(\alpha^2 \rho(\lambda))\right],
  \label{g0t-ba}
  \\
  \ri \Delta(\lambda) \, \wt g_1(-\ri\lambda, t)
  &=
  \left(\alpha-1\right) \rho(\lambda)^2
  \left[\re^{-\ri\lambda t} \, \what u(\gamma(\lambda), t) - \what u_0(\gamma(\lambda))\right]
  \nn\\*
  &\qquad
  -
  \left(\alpha-1\right) \rho(\lambda)^2 \left[\re^{-\ri\lambda t} \, \what u(-\ri\gamma(\lambda), t) - \what u_0(-\ri\gamma(\lambda))\right]
  \nn\\*
  &\qquad
  -
  \left(1+\ri\right) \gamma(\lambda) \left[\alpha \rho(\lambda)^2 + \ri\gamma(\lambda)^2\right]
  \left[\re^{-\ri\lambda t} \, \what v(\rho(\lambda), t) - \what v_0(\rho(\lambda))
  \right]
  \nn\\*
  &\qquad
  +
  \left(1+\ri\right) \gamma(\lambda) \left[\rho(\lambda)^2 + \ri \gamma(\lambda)^2\right]
  \left[\re^{-\ri\lambda t} \, \what v(\alpha^2 \rho(\lambda), t) - \what v_0(\alpha^2 \rho(\lambda))\right],
  \label{g1t-ba}
\end{align}
with the determinant $\Delta(\lambda)$ given by
\begin{equation}
  \Delta(\lambda) =  \left(1+\ri\right) \left(\alpha-1\right) \gamma(\lambda) \rho(\lambda) \left[ -\ri \alpha^2 \gamma(\lambda)^2 + \left(1-\ri\right) \gamma(\lambda) \rho(\lambda) + \alpha  \rho(\lambda)^2 \right].
\end{equation}

\begin{lem} \label{lem:bihS-A.Deltazeros}
  $\Delta(\lambda) = 0$ if and only if $\lambda=0$.
\end{lem}

\begin{proof}
  Recalling the definitions of $\gamma(\lambda)$ and $\rho(\lambda)$, we readily observe that $\lambda=0$ is a zero of $\Delta(\lambda)$.
  By definitions~\eqref{eqn:defn.rho} and~\eqref{eqn:defn.gamma}, $\gamma(\lambda)^2 = \re^{-\ri\frac{\pi}{4}} (\ri\lambda)^{\frac 12}$, $\rho(\lambda)^2 = \re^{-\ri\frac{\pi}{3}} (\ri\lambda)^{\frac 23}$ and $\gamma(\lambda)\rho(\lambda) = \re^{-\ri\frac{7\pi}{24}} (\ri\lambda)^{\frac{7}{12}}$.
  Thus,
  \begin{align}
    -\ri \alpha^2 \gamma(\lambda)^2 + \left(1-\ri\right) \gamma(\lambda) \rho(\lambda) + \alpha  \rho(\lambda)^2
    \hspace{-3em}&\hspace{3em}=
    -\ri \alpha^2 \re^{-\ri\frac{\pi}{4}} (\ri\lambda)^{\frac 12} + \left(1-\ri\right) \re^{-\ri\frac{7\pi}{24}} (\ri\lambda)^{\frac{7}{12}} + \alpha \re^{-\ri\frac{\pi}{3}} (\ri\lambda)^{\frac 23}
    \nn\\
    &=
    (\ri\lambda)^{\frac 12}
    \left[
    -\ri \re^{\ri\frac{4\pi}{3}} \re^{-\ri\frac{\pi}{4}} + \sqrt 2 \, \re^{-\ri\frac{\pi}{4}}  \re^{-\ri\frac{7\pi}{24}} (\ri\lambda)^{\frac{1}{12}} + \re^{\ri\frac{2\pi}{3}} \re^{-\ri\frac{\pi}{3}} (\ri\lambda)^{\frac 16}
    \right]
    \nn\\
    &=
    (\ri\lambda)^{\frac 12} \re^{\ri\frac{7\pi}{12}}
    \left[
    1 - \sqrt 2 \, \re^{-\ri\frac{\pi}{8}} (\ri\lambda)^{\frac{1}{12}} + \re^{-\ri\frac{\pi}{4}} (\ri\lambda)^{\frac 16}
    \right].\label{ml0}
  \end{align}
  Therefore, $\lambda=0$ is a zero of $\Delta(\lambda)$ of order $\frac 14 + \frac 13 + \frac 12 = \frac{13}{12}$.
  Furthermore, the square bracket on the right side of~\eqref{ml0} is of the form $z^2 - \sqrt2 \, z + 1$ with $z=\re^{-\ri\frac{\pi}{8}} (\ri\lambda)^{\frac{1}{12}}$.
  Thus, it vanishes if and only if
  \[
    \re^{-\ri\frac{\pi}{8}} (\ri\lambda)^{\frac{1}{12}} = \frac{\sqrt 2 \pm \ri \sqrt 2}{2}
    = \frac{1}{\sqrt 2} \left(1\pm \ri\right)
    = \re^{\pm \ri\frac \pi 4}
    \ \Leftrightarrow \
    (\ri\lambda)^{\frac{1}{12}} = \re^{\ri\frac{3\pi}{8}} \text{ or } \re^{-\ri\frac \pi 8}
    \ \Rightarrow \
    \ri\lambda = \ri
    .
  \]
  So the only other \emph{possible} zero of $\Delta(\lambda)$ except for $\lambda=0$ is $\lambda = 1$.
  However,
  \[
    \gamma(1) = \re^{-\ri\frac \pi 8} \re^{\ri\frac{\pi/2}{4}} = 1,
    \qquad
    \rho(1) = \re^{-\ri\frac \pi 6} \re^{\ri\frac{\pi/2}{3}} = 1
  \]
  hence
  \[
    \Delta(1) = \left(1+\ri\right) \left(\alpha-1\right) \left[ -\ri \alpha^2 + \left(1-\ri\right) + \alpha \right] =
    \big(1-\sqrt 3\big) \left(\alpha-1\right) \neq 0,
  \]
  allowing us to conclude that the unique zero of $\Delta(\lambda)$ is $\lambda=0$.
\end{proof}

In order to use~\eqref{g0t-ba} and~\eqref{g1t-ba} in integral representations~\eqref{utm-v} and~\eqref{utm-bih}, we make suitable changes of variable in the terms involving, respectively, $\widetilde g_0(-\ri k^3, t)$ and $\widetilde g_0(-\ri k^4, t), \widetilde g_1(-\ri k^4, t)$.
Specifically, as before, in view of the definition of $\rho(\lambda)$ (see also figure~\ref{fig:rho}), the last term in~\eqref{utm-v} requires the change of variable $k = \alpha \rho(\lambda)$, which maps $\lambda \in \RR$ to $k\in\CAiry$.%
\footnote{
  The changes of variable $k = \rho(\lambda)$ and $k = \alpha^2 \rho(\lambda)$ still imply the desired relation $k^3 = \lambda$ but map $\lambda \in \RR$ to the boundary of the sixth and fourth sextant of the complex $k$ plane, respectively, which do not correspond to the contour $\CAiry$ present in~\eqref{utm-v}.
  As a rule of thumb, these latter transformations are not suitable because they introduce exponential growth via $\re^{\ri kx}$.
}
Then,~\eqref{utm-v} takes the form~\eqref{utm-v2}.

Similarly, in the case of~\eqref{utm-bih}, based on definition~\eqref{eqn:defn.gamma} of $\gamma(\lambda)$ (see also figure~\ref{fig:gamma}), in the penultimate term of~\eqref{utm-bih} we let $k= \ri\gamma(\lambda)$ with $\lambda \in \RR$, since this transformation maps $\RR$ to the relevant contour $\CbiSone$.
In the last term of~\eqref{utm-bih} we let $k= -\gamma(\lambda)$ with $\lambda \in \RR$, since this transformation maps $\RR$ to contour $\CbiStwo$.
In fact, being proactive, we first deform the contours $\CbiSone$ and $\CbiStwo$ to the
contours
\[
  \widetilde{\mathcal C}_{4 \Mspacer 1} = \partial \left\{k\in\mathbb C: \frac{\pi}{4}  < \arg(k) < \frac{\pi}{2}, \; \abs k > \epsilon \right\},
  \quad
  \widetilde{\mathcal C}_{4 \Mspacer 2} = \partial \left\{k\in\mathbb C: \frac{3\pi}{4} < \arg(k) < \pi,           \; \abs k > \epsilon \right\},
\]
Because both
circular arcs of radius $R^4$ (with $R>0$)
are mapped to the semicircle
$\left\{R\epsilon \re^{\ri\theta}: -\pi < \theta < 0\right\}$
with anticlockwise orientation, integral representation~\eqref{utm-bih} becomes
\begin{multline} \label{utm-bih2}
  2\pi u(x, t)
  =
  \int_{\RR} \re^{\ri kx+\ri k^4t} \, \what u_0(k) \D k
  -
  \int_{\CbiSone} \re^{\ri kx+\ri k^4t} \left[\left(1+\ri\right) \what u_0(-\ri k) - \ri \, \what u_0(-k) \right] \D k
  \\
  -
  \int_{\CbiStwo} \re^{\ri kx+\ri k^4t} \left[\left(1-\ri\right) \what u_0(\ri k) + \ri \, \what u_0(-k) \right] \D k
  \\
  +
  \frac{1-\ri}{2}
  \int_{\widetilde{\RR}} \re^{-\gamma(\lambda) x+\ri\lambda t}
  \left[\widetilde g_0(-\ri\lambda, t) + \ri \left(\gamma(\lambda)\right)^{-1} \widetilde g_1(-\ri\lambda, t) \right] \D\lambda
  \\
  +
  \frac{1+\ri}{2}
  \int_{\widetilde{\RR}} \re^{-\ri\gamma(\lambda) x+\ri\lambda t}
  \left[ \widetilde g_0(-\ri\lambda, t) - \left(\gamma(\lambda)\right)^{-1} \widetilde g_1(-\ri\lambda, t) \right] \D\lambda.
\end{multline}
By lemma~\ref{lem:bihS-A.Deltazeros}, the contour $\widetilde{\RR}$ bypasses from below the unique zero of $\gamma(\lambda)$ at the origin.

Inserting expression~\eqref{g0t-ba} into representation~\eqref{utm-v2}, we obtain
\begin{align}
  2\pi v(x, t)
  &=
  \int_{\RR} \re^{\ri kx+\ri k^3 t} \, \what v_0(k) \D k
  +
  \int_{\CAiry} \re^{\ri kx+\ri k^3 t} \left[\alpha \, \what v_0(\alpha k) + \alpha^2 \, \what v_0(\alpha^2 k)\right] \D k
  \nn\\
  &\qquad
  +
  \left(\alpha^2-1\right)
  \int_{\widetilde{\RR}} \re^{\ri\alpha \rho(\lambda) x+\ri\lambda t}
  \frac{\rho(\lambda)}{\Delta(\la)}
  \left[\what u_0(-\ri\gamma(\lambda)) - \what u_0(\gamma(\lambda))\right] \D\lambda
  \nn\\
  &\qquad
  +
  \int_{\widetilde{\RR}} \re^{\ri\alpha \rho(\lambda) x+\ri\lambda t}
  \frac{\gamma(\lambda)}{\Delta(\la)}
  \Big\{\left[\alpha^2 \left(1+\ri\right) \rho(\lambda) - 2\gamma(\lambda)\right] \what v_0(\rho(\lambda))
  \nn\\*
  &\hspace{17em}
  -
  \left[\left(1+\ri\right) \rho(\lambda) - 2\gamma(\lambda)\right]
  \what v_0(\alpha^2 \rho(\lambda))
  \Big\} \D\lambda
  \nn\\
  &\qquad
  -
  \left(\alpha^2-1\right)
  \int_{\widetilde{\RR}} \re^{\ri\alpha \rho(\lambda) x}
  \frac{\rho(\lambda)}{\Delta(\la)} \left[\what u(-\ri\gamma(\lambda), t) - \what u(\gamma(\lambda), t)\right] \D\lambda
  \nn\\
  &\qquad
  -
  \int_{\widetilde{\RR}} \re^{\ri\alpha \rho(\lambda) x}
  \frac{\gamma(\lambda)}{\Delta(\la)}
  \Big\{\left[\alpha^2 \left(1+\ri\right) \rho(\lambda) - 2\gamma(\lambda)\right] \what v(\rho(\lambda), t)
  \nn\\*
  &\hspace{17em}
  -
  \left[\left(1+\ri\right) \rho(\lambda) - 2\gamma(\lambda)\right]
  \what v(\alpha^2 \rho(\lambda), t)
  \Big\} \D\lambda.
  \label{eqn:biSAiry.v.withExtraTerms}
\end{align}
Likewise, inserting the expressions~\eqref{g0t-ba} and~\eqref{g1t-ba} into representation~\eqref{utm-bih2}, we have
\begin{align}
  2\pi u(x, t)
  &=
  \int_{\RR} \re^{\ri kx+\ri k^4t} \, \what u_0(k) \D k
  -
  \int_{\CbiSone} \re^{\ri kx+\ri k^4t} \left[\left(1+i\right) \what u_0(-\ri k) - \ri \, \what u_0(-k) \right] \D k
  \nn\\
  &\quad
  -
  \int_{\CbiStwo} \re^{\ri kx+\ri k^4t} \left[\left(1-\ri\right) \what u_0(\ri k) + \ri \, \what u_0(-k) \right] \D k
  \nn\\
  &\quad
  -
  \frac{1-\ri}{2}
  \int_{\widetilde{\RR}} \re^{-\gamma(\lambda) x+\ri\lambda t}
  \left(\alpha-1\right) \rho(\lambda)
  \left[\frac{\rho(\la)}{\gamma(\la)}-\alpha^2\right]
  \left[\what u_0(\gamma(\lambda)) - \what u_0(-\ri\gamma(\lambda))\right]
  \frac{\D\lambda}{\Delta(\la)}
  \nn\\
  &\quad
  -
  \frac{1-\ri}{2}
  \int_{\widetilde{\RR}} \re^{-\gamma(\lambda) x+\ri\lambda t}
  \bigg\{
  \left[
  \left(1+\ri\right) \rho(\lambda) \left(\gamma(\lambda) + \rho(\lambda)\right) - \left(3-\ri\right) \gamma(\lambda)^2
  \right]
  \what v_0(\alpha^2 \rho(\lambda))
  \nn\\*
  &\hspace*{4cm}
  -
  \left[
  \alpha \left(1+\ri\right) \rho(\lambda) \left(\alpha \gamma(\lambda)+\rho(\lambda)\right) -\left(3-\ri\right)\gamma(\lambda)^2
  \right]
  \what v_0(\rho(\lambda))
  \bigg\}
  \frac{\D\lambda}{\Delta(\la)}
  \nn\\
  &\quad
  -
  \frac{1+\ri}{2}
  \int_{\widetilde{\RR}} \re^{-\ri\gamma(\lambda) x+\ri\lambda t}
  \left(\alpha-1\right) \rho(\lambda)
  \left[\frac{\ri\rho(\la)}{\gamma(\la)}-\alpha^2\right]
  \left[\what u_0(\gamma(\lambda)) - \what u_0(-\ri\gamma(\lambda))\right]
  \frac{\D\lambda}{\Delta(\la)}
  \nn\\
  &\quad
  -
  \frac{1+\ri}{2}
  \int_{\widetilde{\RR}} \re^{-\ri\gamma(\lambda) x+\ri\lambda t}
  \bigg\{
  \left[
  \alpha \left(1-\ri\right) \rho(\lambda) \left(\rho(\lambda)-\ri\alpha \gamma(\lambda)\right) + \left(3+\ri\right)\gamma(\lambda)^2
  \right]
  \what v_0(\rho(\lambda))
  \nn\\*
  &\hspace*{4cm}
  -
  \left[
  \left(1-\ri\right)\rho(\lambda) \left(\rho(\lambda)-\ri\gamma(\lambda)\right)
  +
  \left(3+\ri\right) \gamma(\lambda)^2
  \right]
  \what v_0(\alpha^2 \rho(\lambda))
  \bigg\}
  \frac{\D\lambda}{\Delta(\la)}
  \nn\\
  &\quad
  +
  \frac{1-\ri}{2}
  \int_{\widetilde{\RR}} \re^{-\gamma(\lambda) x}
  \left(\alpha-1\right) \rho(\lambda)
  \left[\frac{\rho(\la)}{\gamma(\la)}-\alpha^2\right]
  \left[ \what u(\gamma(\lambda), t) - \what u(-\ri\gamma(\lambda), t)\right]
  \frac{\D\lambda}{\Delta(\la)}
  \nn\\
  &\quad
  +
  \frac{1-\ri}{2}
  \int_{\widetilde{\RR}} \re^{-\gamma(\lambda) x}
  \bigg\{
  \left[
  \left(1+\ri\right) \rho(\lambda) \left(\gamma(\lambda) + \rho(\lambda)\right) - \left(3-\ri\right) \gamma(\lambda)^2
  \right]
  \what v(\alpha^2 \rho(\lambda), t)
  \nn\\*
  &\hspace*{3.5cm}
  -
  \left[
  \alpha \left(1+\ri\right) \rho(\lambda) \left(\alpha \gamma(\lambda)+\rho(\lambda)\right) -\left(3-\ri\right)\gamma(\lambda)^2
  \right]
  \what v(\rho(\lambda), t)
  \bigg\}
  \frac{\D\lambda}{\Delta(\la)}
  \nn\\
  &\quad
  +
  \frac{1+\ri}{2}
  \int_{\widetilde{\RR}} \re^{-\ri\gamma(\lambda) x}
  \left(\alpha-1\right) \rho(\lambda)
  \left[\frac{\ri\rho(\la)}{\gamma(\la)}-\alpha^2\right]
  \left[ \what u(\gamma(\lambda), t) - \what u(-\ri\gamma(\lambda), t)\right]
  \frac{\D\lambda}{\Delta(\la)}
  \nn\\
  &\quad
  +
  \frac{1+\ri}{2}
  \int_{\widetilde{\RR}} \re^{-\ri\gamma(\lambda) x}
  \bigg\{
  \left[
  \alpha \left(1-\ri\right) \rho(\lambda) \left(\rho(\lambda)-\ri\alpha \gamma(\lambda)\right) + \left(3+\ri\right)\gamma(\lambda)^2
  \right]
  \what v(\rho(\lambda), t)
  \nn\\*
  &\hspace*{3.5cm}
  -
  \left[
  \left(1-\ri\right)\rho(\lambda) \left(\rho(\lambda)-\ri\gamma(\lambda)\right)
  +
  \left(3+\ri\right) \gamma(\lambda)^2
  \right]
  \what v(\alpha^2 \rho(\lambda), t)
  \bigg\}
  \frac{\D\lambda}{\Delta(\la)}.
  \label{eqn:biSAiry.u.withExtraTerms}
\end{align}
It remains to show that the latter two integrals on the right of~\eqref{eqn:biSAiry.v.withExtraTerms} and the final four integrals on the right of~\eqref{eqn:biSAiry.u.withExtraTerms} are equal to zero.

From definition~\eqref{eqn:defn.gamma}, or easily observed from figure~\ref{fig:gamma}, if $\Im(\lambda) \leq 0$ then
\begin{equation} \label{eqn:biSAiry.im-ineq}
  \Im(\gamma(\lambda)) \leq 0,
  \quad
  \Im(\ri\gamma(\lambda)) \geq \frac1{\sqrt2}\abs\la \geq0,
  \quad
  \Im(-\gamma(\lambda)) \geq 0,
  \quad
  \Im(-\ri\gamma(\lambda)) \leq \frac{-1}{\sqrt2}\abs\la \leq0,
\end{equation}
while the latter three of inequalities~\eqref{im-ineq} remain relevant.

\begin{lem} \label{lem:biSAiry.extraTermsVanish}
  All terms on the right of equations~\eqref{eqn:biSAiry.v.withExtraTerms} and~\eqref{eqn:biSAiry.u.withExtraTerms} which contain $\widehat u(\argdot,t)$ or $\widehat v(\argdot,t)$ evaluate to zero.
\end{lem}

\begin{proof}
  Note first that, provided $\abs\la>2^{12}$
  \begin{align*}
    \abs{\Delta(\la)} &= \sqrt6 \abs{\gamma(\la)\rho(\la)} \abs{ -\ri\alpha^2\gamma(\la)^2 + (1-\ri)\gamma(\la)\rho(\la) + \alpha\rho(\la)^2 } \\
    &\geq \sqrt6 \abs\la^{\frac14+\frac13}\left( \abs\la^{\frac23} - \sqrt2\abs\la^{\frac14+\frac13} - \abs\la^{\frac12} \right) \\
    \geq \abs\la^{\frac54}.
  \end{align*}

  Integrating by parts, we find that
  \[
    \widehat u(-\ri\gamma(\la),t) = \frac{u(0,t)}{\gamma(\la)} + \frac1{\gamma(\la)} \int_0^\infty \re^{-\gamma(\la)y}u_x(y,t)\D y,
  \]
  so
  \begin{align*}
    \abs{\widehat u(-\ri\gamma(\la),t)}
    &\leq \frac1{\abs{\gamma(\la)}} \left( \abs{u(0,t)} + \int_0^\infty \re^{y\Im(-\ri\gamma(\la))} \abs{u_y(y,t)} \D y \right) \\
    &\leq \frac1{\abs{\gamma(\la)}} \left( \normp{u(t)}{\Lebesgue^\infty_x(0,\infty)} + \normp{u_x(t)}{\Lebesgue^1_x(0,\infty)} \right),
  \end{align*}
  where we have used (the weaker version of) the latter of inequalities~\eqref{eqn:biSAiry.im-ineq} to justify the final inequality.
  Similarly,
  \begin{align*}
    \abs{\widehat u(\gamma(\la),t)}
    &\leq
    \frac1{\abs{\gamma(\la)}} \left( \normp{u(t)}{\Lebesgue^\infty_x(0,\infty)} + \normp{u_x(t)}{\Lebesgue^1_x(0,\infty)} \right),
    \\
    \abs{\widehat v(\rho(\la),t)}
    &\leq
    \frac1{\abs{\rho(\la)}} \left( \normp{v(t)}{\Lebesgue^\infty_x(0,\infty)} + \normp{v_x(t)}{\Lebesgue^1_x(0,\infty)} \right),
    \\
    \abs{\widehat v(\alpha^2\rho(\la),t)}
    &\leq
    \frac1{\abs{\rho(\la)}} \left( \normp{v(t)}{\Lebesgue^\infty_x(0,\infty)} + \normp{v_x(t)}{\Lebesgue^1_x(0,\infty)} \right).
  \end{align*}
  The above inequalities may be summarised as, uniformly in $\theta\in[-\pi,0]$,
  \begin{align*}
    \widehat u\left(\gamma\left(R\re^{\ri\theta}\right),t\right),\; \widehat u\left(-\ri\gamma\left(R\re^{\ri\theta}\right),t\right) &= \bigoh{R^{\frac{-1}4}},
    & &\mbox{as } R\to\infty \\
    \widehat v\left(\rho\left(R\re^{\ri\theta}\right),t\right),\; \widehat v\left(\alpha^2\rho\left(R\re^{\ri\theta}\right),t\right) &= \bigoh{R^{\frac{-1}3}},
    & &\mbox{as } R\to\infty.
  \end{align*}

  For $R$ greater than the deformation in $\widetilde\RR$ away from zero, let $\mathcal C_R^-$ be the contour defined by equation~\eqref{eqn:lSAiry.defnCRminus}.
  By Cauchy's theorem applied to the region bounded between $\widetilde\RR$ and $\mathcal C_R^-$,
  \[
    \int_{\widetilde\RR} \re^{\ri\alpha\rho(\la)x} \frac{\rho(\la)}{\Delta(\la)} \widehat u(-\ri\gamma(\la),t) \D\la = \lim_{R\to\infty}\int_{\mathcal C_R^-} \re^{\ri\alpha\rho(\la)x} \frac{\rho(\la)}{\Delta(\la)} \widehat u(-\ri\gamma(\la),t) \D\la,
  \]
  and equivalent contour deformation identities hold for all the other terms of interest.
  Therefore, it is sufficient to show that the limit on the right is zero and equivalent limits for the other terms are each zero.

  For $R>0$ sufficiently large, we use the above inequalities to bound
  \begin{align}
    \notag
    \abs{ \int_{\mathcal C_R^-} \re^{\ri\alpha\rho(\la)x} \frac{\rho(\la)}{\Delta(\la)} \widehat u(-\ri\gamma(\la),t) \D\la }
    &\leq
    \abs{ \int_{-\pi}^0 \re^{\ri\alpha\rho\left(R\re^{\ri\theta}\right)x} \frac{\rho\left(R\re^{\ri\theta}\right)}{\Delta\left(R\re^{\ri\theta}\right)} \widehat u\left(-\ri\gamma\left(R\re^{\ri\theta}\right),t\right) R \D \theta } \\
    \label{eqn:biSAiry.extraTermsVanish.Proof1}
    &\leq
    \int_{-\pi}^0 \abs{\re^{\ri\alpha\rho\left(R\re^{\ri\theta}\right)x}} \abs{\frac{\rho\left(R\re^{\ri\theta}\right)}{\Delta\left(R\re^{\ri\theta}\right)}} \abs{\widehat u\left(-\ri\gamma\left(R\re^{\ri\theta}\right),t\right)} R \D \theta \\
    \notag
    &\leq
    \frac{R^{\frac13}}{R^{\frac54}}\bigoh{R^{\frac{-1}4}} R \int_{-\pi}^0 \re^{-x\Im\left( \alpha\rho\left( R\re^{\ri\theta} \right) \right)} \D\theta \\
    \notag
    &\leq
    \bigoh{R^{\frac{-1}6}} \int_{-\pi}^0 \re^{-xR^{\frac13} \Im\left( \alpha\rho\left( \re^{\ri\theta} \right) \right)} \D\theta.
  \end{align}
  Now, for $\phi = (\pi-\theta)/3$,
  \(
    \Im\left( \alpha\rho\left( \re^{\ri\theta} \right) \right)
    = \sin(\phi),
  \)
  so
  \[
    \int_{-\pi}^0 \re^{-xR^{\frac13} \Im\left( \alpha\rho\left( \re^{\ri\theta} \right) \right)} \D\theta
    =
    3 \int_{\frac\pi3}^{\frac{2\pi}3} \re^{-xR^{\frac13} \sin(\phi)} \D\phi.
  \]
  To obtain a bound on this integral, we use the convexity inequality
  \[
    \frac\pi3 \leq \phi \leq \frac{2\pi}3
    \Rightarrow
    \frac{3\sqrt3}{4\pi} \phi \leq \sin(\phi)
  \]
  to obtain
  \[
    3 \int_{\frac\pi3}^{\frac{2\pi}3} \re^{-xR^{\frac13} \sin(\phi)} \D\phi
    \leq
    3 \int_{\frac\pi3}^{\frac{2\pi}3} \re^{-xR^{\frac13} \frac{3\sqrt3}{4\pi} \phi} \D\phi
    \leq
    R^{\frac{-1}3} \frac{4\pi}{x\sqrt3} \left[ \re^{-xR^{\frac13} \frac{\sqrt3}{2}} - \re^{-xR^{\frac13} \frac{\sqrt3}{4} \phi} \right].
  \]
  Hence, overall,
  \[
    \abs{ \int_{\mathcal C_R^-} \re^{\ri\alpha\rho(\la)x} \frac{\rho(\la)}{\Delta(\la)} \widehat u(-\ri\gamma(\la),t) \D\la }
    = \bigoh{R^{\frac{-1}2}},
  \]
  so
  \[
    \int_{\widetilde\RR} \re^{\ri\alpha\rho(\la)x} \frac{\rho(\la)}{\Delta(\la)} \widehat u(-\ri\gamma(\la),t) \D\la = 0.
  \]

  We could have used the stronger convexity bound $\frac{\sqrt3}2\leq\sin(\phi)$ when evaluating the final integral and obtained faster decay, but analogous bounds are not available for all of the integrals from equation~\eqref{eqn:biSAiry.v.withExtraTerms}, so we have given the weaker but more broadly adaptable version of the argument.
  The above argument can be summarised as the application of the asymptotic bounds
  \begin{align*}
    \int_{-\pi}^0 \abs{\re^{\ri\alpha\rho\left(R\re^{\ri\theta}\right)x}} \D\theta &= \bigoh{R^{\frac{-1}3}}, \\
    \abs{\frac{\rho\left(R\re^{\ri\theta}\right)}{\Delta\left(R\re^{\ri\theta}\right)}} &= \bigoh{R^{\frac13}}\bigoh{R^{\frac{-5}4}},
    & \mbox{ uniformly in } \theta\in[-\pi,0], \\
    \abs{\widehat u\left(-\ri\gamma\left(R\re^{\ri\theta}\right),t\right)} &= \bigoh{R^{\frac{-1}4}},
    & \mbox{ uniformly in } \theta\in[-\pi,0], \\
    R &= \bigoh{R},
  \end{align*}
  to the integral~\eqref{eqn:biSAiry.extraTermsVanish.Proof1}.
  It is straightforward to analyse the other integrals via analogous bounds on their corresponding components.
\end{proof}

With lemma~\ref{lem:biSAiry.extraTermsVanish}, we arrive at the effective solution representation
\begin{multline} \label{eqn:biSAiry.soln.v}
  2\pi v(x, t)
  =
  \int_{\RR} \re^{\ri kx+\ri k^3 t} \, \what v_0(k) \D k
  +
  \int_{\CAiry} \re^{\ri kx+\ri k^3 t} \left[\alpha \, \what v_0(\alpha k) + \alpha^2 \, \what v_0(\alpha^2 k)\right] \D k
  \\
  +
  \left(\alpha^2-1\right)
  \int_{\widetilde{\RR}} \re^{\ri\alpha \rho(\lambda) x+\ri\lambda t}
  \frac{\rho(\lambda)}{\Delta(\la)}
  \left[\what u_0(-\ri\gamma(\lambda)) - \what u_0(\gamma(\lambda))\right] \D\lambda
  \\
  +
  \int_{\widetilde{\RR}} \re^{\ri\alpha \rho(\lambda) x+\ri\lambda t}
  \frac{\gamma(\lambda)}{\Delta(\la)}
  \Big\{\left[\alpha^2 \left(1+\ri\right) \rho(\lambda) - 2\gamma(\lambda)\right] \what v_0(\rho(\lambda))
  -
  \left[\left(1+\ri\right) \rho(\lambda) - 2\gamma(\lambda)\right]
  \what v_0(\alpha^2 \rho(\lambda))
  \Big\} \D\lambda
\end{multline}
and
\begin{align}
  2\pi u(x, t)
  \hspace{-3em} &\hspace{3em} =
  \int_{\RR} \re^{\ri kx+\ri k^4t} \, \what u_0(k) \D k
  -
  \int_{\CbiSone} \re^{\ri kx+\ri k^4t} \left[\left(1+i\right) \what u_0(-\ri k) - \ri \, \what u_0(-k) \right] \D k
  \nn\\
  &\quad
  -
  \int_{\CbiStwo} \re^{\ri kx+\ri k^4t} \left[\left(1-\ri\right) \what u_0(\ri k) + \ri \, \what u_0(-k) \right] \D k
  \nn\\
  &\quad
  -
  \frac{1-\ri}{2}
  \int_{\widetilde{\RR}} \re^{-\gamma(\lambda) x+\ri\lambda t}
  \left(\alpha-1\right) \rho(\lambda)
  \left[\frac{\rho(\la)}{\gamma(\la)}-\alpha^2\right]
  \left[\what u_0(\gamma(\lambda)) - \what u_0(-\ri\gamma(\lambda))\right]
  \frac{\D\lambda}{\Delta(\la)}
  \nn\\
  &\quad
  -
  \frac{1-\ri}{2}
  \int_{\widetilde{\RR}} \re^{-\gamma(\lambda) x+\ri\lambda t}
  \bigg\{
  \left[
  \left(1+\ri\right) \rho(\lambda) \left(\gamma(\lambda) + \rho(\lambda)\right) - \left(3-\ri\right) \gamma(\lambda)^2
  \right]
  \what v_0(\alpha^2 \rho(\lambda))
  \nn\\*
  &\hspace*{8em}
  -
  \left[
  \alpha \left(1+\ri\right) \rho(\lambda) \left(\alpha \gamma(\lambda)+\rho(\lambda)\right) -\left(3-\ri\right)\gamma(\lambda)^2
  \right]
  \what v_0(\rho(\lambda))
  \bigg\}
  \frac{\D\lambda}{\Delta(\la)}
  \nn\\
  &\quad
  -
  \frac{1+\ri}{2}
  \int_{\widetilde{\RR}} \re^{-\ri\gamma(\lambda) x+\ri\lambda t}
  \left(\alpha-1\right) \rho(\lambda)
  \left[\frac{\ri\rho(\la)}{\gamma(\la)}-\alpha^2\right]
  \left[\what u_0(\gamma(\lambda)) - \what u_0(-\ri\gamma(\lambda))\right]
  \frac{\D\lambda}{\Delta(\la)}
  \nn\\
  &\quad
  -
  \frac{1+\ri}{2}
  \int_{\widetilde{\RR}} \re^{-\ri\gamma(\lambda) x+\ri\lambda t}
  \bigg\{
  \left[
  \alpha \left(1-\ri\right) \rho(\lambda) \left(\rho(\lambda)-\ri\alpha \gamma(\lambda)\right) + \left(3+\ri\right)\gamma(\lambda)^2
  \right]
  \what v_0(\rho(\lambda))
  \nn\\*
  &\hspace*{8em}
  -
  \left[
  \left(1-\ri\right)\rho(\lambda) \left(\rho(\lambda)-\ri\gamma(\lambda)\right)
  +
  \left(3+\ri\right) \gamma(\lambda)^2
  \right]
  \what v_0(\alpha^2 \rho(\lambda))
  \bigg\}
  \frac{\D\lambda}{\Delta(\la)}.
  \label{eqn:biSAiry.soln.u}
\end{align}
Proposition~\ref{prop:biS-Airy} follows from equations~(\ref{eqn:biSAiry.soln.v}--\ref{eqn:biSAiry.soln.u}) via change of variables~\eqref{eqn:w-to-u-change-of-variable}.

\section{Numerical evaluation via Filon quadrature} \label{sec:numerics}

Computing the integrals in propositions~\ref{prop:lS-lKdV}--\ref{prop:biS-Airy} presents a numerical challenge due to the highly oscillatory nature of the integrands.
We can make use of a Filon type quadrature method~\cite{DB2008a,Fil1930a}, because we can separate the integrand into a slowly varying function, $f(k)$, multiplied by a potentially highly oscillatory function of the form $\re^{\ri\psi(k)}$ where $\psi(k)$ is the phase function and $\psi'(k)$ is slowly varying.
Specifically, we compute an integral of the form
\[
  I=\int_{a}^{b}f(x)\re^{\ri\psi(x)}\,\D x,
\]
where $f(x)$ and $\psi'(x)$ are slowly varying functions, by splitting the domain of integration at the points $a=x_0,x_1,\ldots,x_N=b$ and approximating the functions $f(x)$ and $\psi(x)$ by piecewise linear functions
\begin{multline*}
  I=\Delta x_0\left(\frac{\ri \re^{\ri \psi_0}}{\psi_{1}-\psi_0}-\frac{\re^{\ri \psi_{1}}-\re^{\ri \psi_0}}{(\psi_{1}-\psi_0)^2}\right)f_0+\sum_{i=1}^{N-1}\left[\Delta x_{i-1}\left(-\frac{\ri \re^{\ri \psi_{i}}}{\psi_{i}-\psi_{i-1}}+\frac{\re^{\ri \psi_{i}}-\re^{\ri \psi_{i-1}}}{(\psi_{i}-\psi_{i-1})^2}\right)\right.\\
  \left.+\Delta x_i\left(\frac{\ri \re^{\ri \psi_i}}{\psi_{i+1}-\psi_i}-\frac{\re^{\ri \psi_{i+1}}-\re^{\ri \psi_i}}{(\psi_{i+1}-\psi_i)^2}\right)\right]f_i+\Delta x_{N-1}\left(-\frac{\ri \re^{\ri \psi_{N}}}{\psi_{N}-\psi_{N-1}}+\frac{\re^{\ri \psi_{N}}-\re^{\ri \psi_{N-1}}}{(\psi_{N}-\psi_{N-1})^2}\right)f_{N}.
\end{multline*}

As the transformed initial datum can be highly oscillatory, we consider each term of each integral separately.
Using the last term of the final integral in the $w$ formula from proposition~\ref{prop:lS-Airy} as an example, we assume that the initial datum is localised about some point, say $x_0$, and, taking inspiration from the Fourier transform on the infinite line, we write $\what v_0(\alpha^2\rho(k)) = V(\alpha^2\rho(k))\re^{-\ri \alpha^2\rho(k) x_0}$ where $V(\alpha^2\rho(k))$ is assumed to be slowly varying.
Thus for our example we have
\[
  f(k) = \frac{V(\alpha^2\rho(k))}{\left(\alpha-1\right)\rho(k) \left[\rho(k) + \alpha^2 \sigma(k)\right]},
  \qquad\text{and}\qquad
  \psi(k) = \sigma(k)x+kt-\alpha^2\rho(k) x_0.
\]

Taking the new form of our integrand, we then parametrise our integration contour by a series of straight line arcs
\[
  I = \left(\int_{-R_\infty}^{-\epsilon}+\int_{-\epsilon}^{-\ri \epsilon}+\int_{-\ri \epsilon}^{\epsilon}+\int_\epsilon^{R_\infty}\right)f(k)\re^{\ri \psi(k)}\,\D k,
\]
where $R_\infty$ is the truncation of infinity and $\epsilon$ is some small number.
For our calculation of this integral, we take $R_\infty=40^3$ and $\epsilon=1$.
We acknowledge that $\psi(k)$ will not be purely real along every arc of this contour, but, provided care is taken when numerically evaluating the weights of integration, there is no issue.

\subsection*{Linear Schr\"odinger and Airy}

The solution of problem~\eqref{eqn:LS},~\eqref{eqn:Airy},~\eqref{eqn:IfC.1} determined in proposition~\ref{prop:lS-Airy} is plotted at time $t=0.39$ in figure~\ref{fig:lS-Airy.w-data} for initial datum supported on the negative (Schr\"odinger) half line, and in figure~\ref{fig:lS-Airy.v-data} for initial datum supported on the positive (Airy) half line.

\begin{figure}
  \centering
  \includegraphics{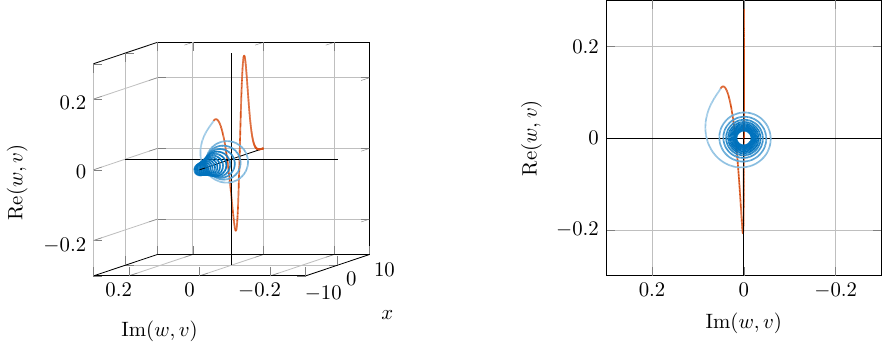}
  \caption{Solution of the linear Schr\"odinger and Airy problem for real initial data supported on the negative (Schr\"odinger) half line, evaluated at $t=0.39$. In blue (faded as $x\to0$) is the complex solution of the linear Schr\"odinger equation and in red is the complex solution of the Airy equation.}
  \label{fig:lS-Airy.w-data}
\end{figure}

\begin{figure}
  \centering
  \includegraphics{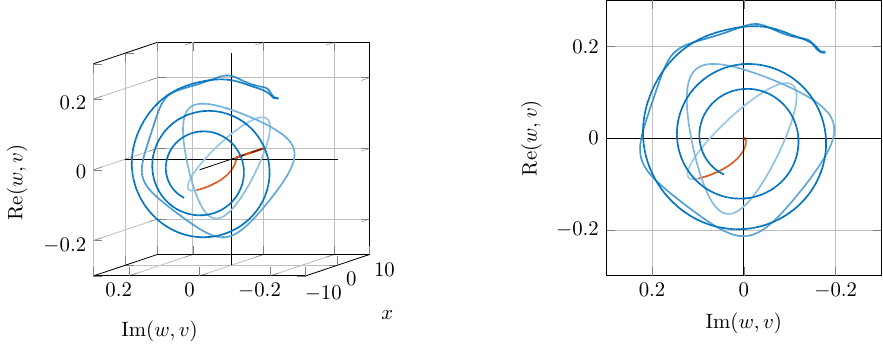}
  \caption{Solution of the linear Schr\"odinger and Airy problem for real initial data supported on the positive (Airy) half line, evaluated at $t=0.39$. In blue is the complex solution of the linear Schr\"odinger equation and in red is the complex solution of the Airy equation.}
  \label{fig:lS-Airy.v-data}
\end{figure}

While the linear Schr\"odinger equation is a bidirectional dispersive equation, the Airy equation is unidirectional and, for our choice of coefficient of $u_{xxx}$, admits waves propogating from larger to smaller positive $x$, towards the interface.
In figure~\ref{fig:lS-Airy.w-data} as the (approximately) real Airy dispersive waves are propogated across the interface, and their governing equation becomes linear Schr\"odinger, they begin to exhibit the expected helical form.
That the solution for $x>0$ is not purely real, despite the real inital datum, is a result of the imposition of the interface condition and the fact that the Airy equation allows some information to travel against the direction of the dispersive waves.
In figure~\ref{fig:lS-Airy.v-data}, as the (approximately) helical dispersive waves of the linear Schr\"odinger equation arrive at the boundary, they are unable to propogate in the Airy regime.
The interface conditions cause the Airy part of the solution to be nonzero, but with no way of carrying the waves into $x>0$, the information can only spread, and slowly so.
This introduces a drag on the linear Schr\"odinger solution for small negative $x$, manifested as the roulette like patterns in the figure.

\subsection*{Other problems}

\begin{figure}
  \centering
  \includegraphics{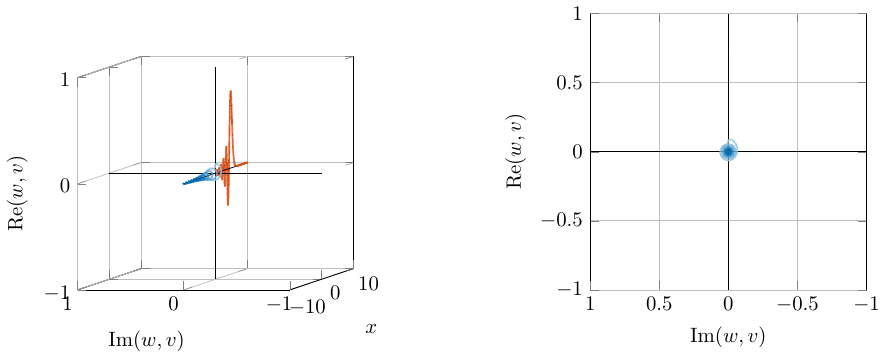}

  \includegraphics{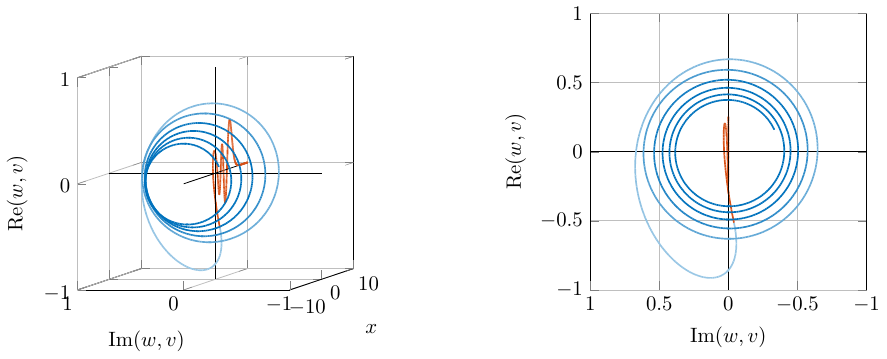}
  \caption{Solution of the biharmonic Schr\"odinger (blue) and Airy (red) problem for real initial data supported on the negative (biharmonic Schr\"odinger) half line, evaluated at $t=0.013$ (top) and $t=0.078$ (bottom).}
  \label{fig:bihS-Airy.w-data}
\end{figure}

\begin{figure}
  \centering
  \includegraphics{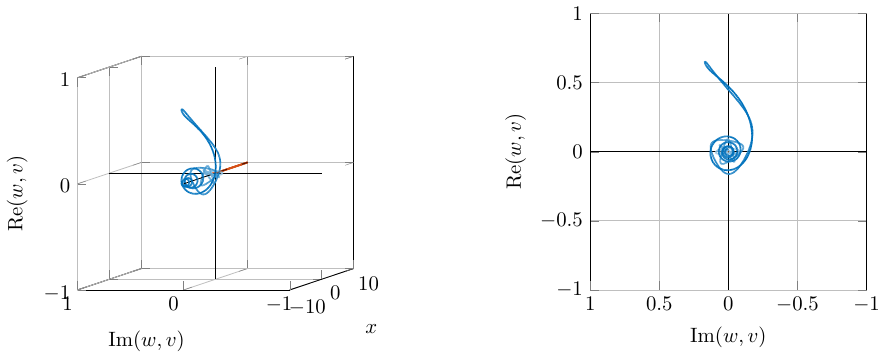}

  \includegraphics{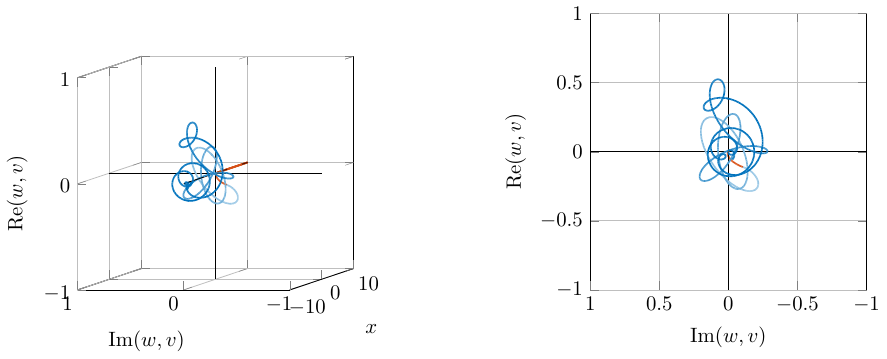}
  \caption{Solution of the biharmonic Schr\"odinger (blue) and Airy (red) problem for real initial data supported on the positive (Airy) half line, evaluated at $t=0.002$ (top) and $t=0.013$ (bottom).}
  \label{fig:bihS-Airy.v-data}
\end{figure}

The solution of the biharmonic Schr\"odinger and Airy problem of proposition~\ref{prop:biS-Airy} is displayed in figure~\ref{fig:bihS-Airy.w-data} for initial datum supported on the positive (Airy) half line, and figure~\ref{fig:bihS-Airy.v-data} for initial datum supported on the negative (biharmonic Schr\"odinger) half line.
These plots exhibit similar features to those of the linear Schr\"odinger and Airy problem, except that the roulette like behaviour appears at much earlier times due to the quartic dispersion relation.

The solution of the heat and Airy problem  of proposition~\ref{prop:heat-Airy} is displayed in figure~\ref{fig:heat-Airy}.
On the right plots, we see that the incoming wave generated by initial datum supported on the right half line disperses on the left half line.
The plots on the left demonstrate exponential decay on the right half line of datum supported initially on the left half line.

\begin{figure}
  \centering
  \includegraphics{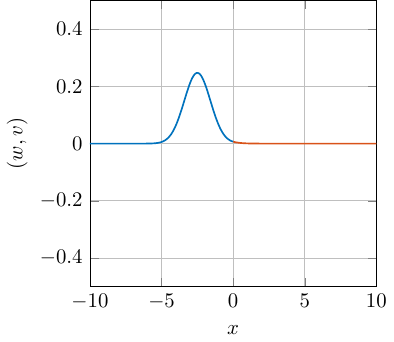} \hspace{2em}
  \includegraphics{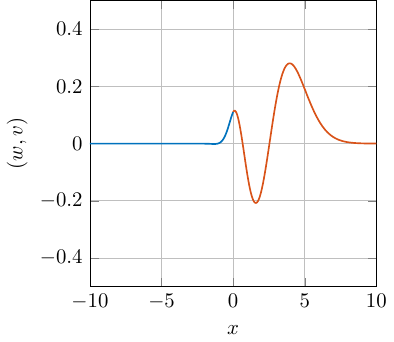}
  \vspace{2ex}

  \includegraphics{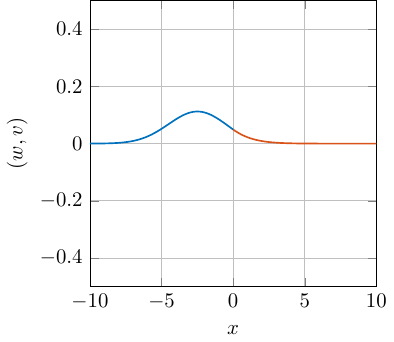} \hspace{2em}
  \includegraphics{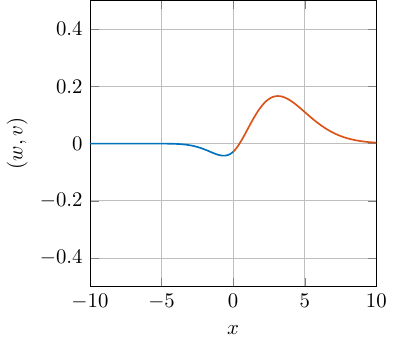}
  \caption{Solution of the heat (blue) and Airy (red) problem for real initial data supported on the negative (heat) half line (left) and positive (Airy) half line (right), evaluated at $t=0.39$ (top) and $t=2$ (bottom).}
  \label{fig:heat-Airy}
\end{figure}

\section*{Acknowledgement}

DM gratefully acknowledges support from the U.S. National Science Foundation (grants NSF-DMS 2206270 and NSF-DMS 2509146) and the Simons Foundation (award SFI-MPS-TSM-00013970).

\bibliographystyle{amsplain}
{\small\bibliography{dbrefs}}

\end{document}